\documentclass[noams]{compositio}
\usepackage{amsmath,amsfonts,amssymb,amscd}
\usepackage[all]{xypic}

\newtheorem{thm}{Theorem}[subsection]

\newtheorem{lem}[thm]{Lemma}
\newtheorem{prop}[thm]{Proposition}
\theoremstyle{definition}

\newtheorem{cor}[thm]{Corollary}
\theoremstyle{remark}
\newtheorem{rmk}[thm]{Remark}

\newtheorem{ex}[thm]{Example}

\def\del{\partial}

\def\R{\mathbb{R}}
\def\C{\mathbb{C}}
\def\Z{\mathbb{Z}}

\def\on{\operatorname}
\def\CC{\mathcal C}
\def\cD{\mathcal D}
\def\CE{\mathcal E}
\def\CF{\mathcal F}
\def\CG{\mathcal G}
\def\CH{\mathcal H}
\def\CK{\mathcal K}
\def\CL{\mathcal L}
\def\CN{\mathcal N}
\def\CP{\mathcal P}
\def\CQ{\mathcal Q}
\def\CO{\mathcal O}
\def\CR{\mathcal R}
\def\CS{\mathcal S}
\def\CT{\mathcal T}

\def\ol{\overline}
\def\intHom{{\mathcal Hom}}

\def\CB{\mathcal B}

\def\hom{{hom}}

\def\CC{\mathcal C}
\def\Sym{\on{Sym}}

\def\fU{\mathfrak U}

\def\CM{\mathcal M}
\def\b{\mathfrak b}

\def\g{\mathfrak g}

\def\fL{\mathfrak L}

\def\fP{\mathfrak P}

\def\fu{\mathfrak u}
\def\fh{\mathfrak h}

\def\wt{\widetilde}

\def\vareps{{\varepsilon}}
\def\CE{\mathcal E}

\def\risom{\stackrel{\sim}{\to}}

\def\wt{\widetilde}

\def\Sh{Sh}
\def\Top{Top}

\def\hra{\hookrightarrow}

\def\risom{\stackrel{\sim}{\to}}

\def\Ch{Ch}

\def\Vect{Vect}

\def\w{\wedge}

\def\inv{/ \hspace{-0.25em} /}

\def\interiorsymbol {\on{int}}

\def\Bun{\on{Bun}}

\def\Spec{\on{Spec}}

\def\oo{\infty}

\begin{document}

\title{Springer theory via the Hitchin fibration}
\author{David Nadler}
\email{nadler@math.northwestern.edu}
\address{Department of Mathematics, Northwestern University,
2033 Sheridan Road, Evanston, IL  60208}



\begin{abstract} 
We develop the Springer theory of Weyl group
representations in the language of symplectic topology.
Given a semisimple complex group $G$,
we describe a Lagrangian brane 
in the cotangent bundle of the adjoint quotient $\g/G$
that produces the perverse sheaves of Springer theory.
The main technical tool is an analysis of the Fourier transform
for constructible sheaves from the perspective of the Fukaya category.
Our results can be viewed as a toy model of the quantization
of Hitchin fibers in the Geometric Langlands program.
\end{abstract}

\maketitle

%


\section{Introduction}

The primary aim of this paper
is to develop Springer's theory~\cite{Spr76, Spr78, Spr82}
of Weyl group
representations in the language of symplectic topology.
Let $G$ be a semisimple complex group with Lie algebra $\g$,
nilpotent cone $\CN\subset \g$,  
Springer resolution 
$\mu:\tilde\CN \to \CN,$
and Weyl group $W$.
In a form due to
Lusztig~\cite{L} and Borho-MacPherson~\cite{BoMac} (building on the topological setting introduced by Kazhdan-Lusztig~\cite{KL}),
Springer theory identifies the group algebra $\C[W]$
with the (degree zero) endomorphisms of the perverse sheaf $S_\CN=R\mu_*\C_{\tilde\CN}[\dim_\C \CN]$.
We explain here how one can tell this story via 
 the Fukaya category
of the cotangent bundle of the adjoint quotient $\g/G$.
An initial motivation for this is the relative accessibility of objects of the Fukaya category
(smooth Lagrangian submanifolds with structure) versus objects such as perverse
sheaves (complexes of sheaves living on a singular variety). For example,
the object of the Fukaya category corresponding to $S_\CN$ (or to be precise,
its Fourier transform) is a regular fiber of a particular instance of the
Hitchin fibration. The Weyl group action arises from Hamiltonian isotopies
coming from motions of the regular parameter.

The approach to Springer theory adopted here is via the Fourier transform for constructible sheaves.
According to Ginzburg~\cite{Ginz} and Hotta-Kashiwara~\cite{HK},
the Fourier transform of the Springer sheaf  $S_\CN$ can be identified with the intersection
cohomology of $\g$ with coefficients in the regular $\C[W]$ local system over the regular
semisimple locus $\g_{rs}\subset \g$. 
(Here we have identified $\g$ and its dual $\g^*$ using the Killing form.)
In general, given a finite-dimensional real vector space $V$, 
the Fourier transform exchanges conic constructible sheaves on $V$ and its dual $V^*$.
The Fourier transform $\CF^\w$
of a constructible sheaf $\CF$
 encodes the Morse groups (or vanishing cycles) of $\CF$ along rays through the origin.
When $V$ is a complex vector space, the (shifted) Fourier transform
exchanges conic perverse sheaves on $V$ and its dual $V^*$.
Its structure is most transparent from the perspective of $\cD$-modules
via the Riemann-Hilbert correspondence.
The Fourier transform for $\cD$-modules exchanges conic $\cD_V$-modules with conic 
 $\cD_{V^*}$-modules by the elementary change of variables
$
x\mapsto -\del_x$,
$
\del_x\mapsto x
$,
where $x$ is a coordinate on $V$, and $\del_x$ is the dual coordinate on $V^*$.
In other words, it is nothing more than a ``$90^\circ$ rotation" of the phase space.

There is a dictionary between
constructible sheaves on a manifold and the Fukaya category
of its cotangent bundle (see~\cite{NZ, N}).
The main technical result of this paper describes how the Fourier transform
for constructible sheaves on $V$ and $V^*$
appears from the perspective of the Fukaya category of $V\times V^*$.
We can identify  $V\times V^*$ as a symplectic target with $T^*V$
and also with $T^*V^*$. (If $x, \xi$ denote coordinates on $T^*V$,
$v, \lambda$ coordinates on $V\times V^*$, and $y, \eta$ coordinates on $T^*V^*$,
then our identifications take the form $x = v = -\eta$, $\xi = \lambda = y$.)
In this way, one can associate to a Lagrangian brane $L\hra V\times V^*$ constructible 
 sheaves $\pi_V(L)$ and $\pi_{V^*}(L)$ living on $V$ and $V^*$ respectively.
We introduce a class of Lagrangian branes in $V\times V^*$, called balanced branes,
such that the conic limits of  $\pi_V(L)$ and $\pi_{V^*}(L)$ 
are Fourier tranforms of each other.
Thus for balanced branes, the operation underlying the Fourier transform
is again 
nothing more than a ``$90^\circ$ rotation"  of the target.
With this theory in hand,
we construct the Fourier transform of the Springer sheaf 
 $S_\CN$ by identifying the corresponding balanced brane
 in $T^*\g$. 

The primary motivation for this paper is the study of the Hitchin fibration
within the framework of the Geometric Langlands program (see Beilinson-Drinfeld~\cite{BD}
and Kapustin-Witten~\cite{KW}).
Fix a smooth projective complex curve $C$.
Let $ \g^*\inv G =\Spec(\Sym \g)^G$ 
be the affine coadjoint quotient, 
$( \g^*\inv G)_{\omega_C}$ its twist by the canonical
bundle of $C$, and 
$\mathbb B_G(C) = \Gamma(C, (\g^*\inv G)_{\omega_C})$ the Hitchin base,
Consider the moduli $\Bun_G(C)$
of $G$-bundles over $C$, its cotangent bundle $T^*\Bun_G(C)$ with fibers
$$
T_\CP^*\Bun_G(C) =\Gamma(C, \g^*_{\CP} \otimes \omega_C),
\quad
\mbox{ for $\CP\in \Bun_G(C)$},
$$
where $\g^*_{\CP}$ denotes the $\CP$-twist of $\g^*$, and  the Hitchin fibration
$$
\CH:T^*\Bun_G(C)\to \mathbb B_G(C)
\qquad
\CH(\CP, \Phi) = \ol \Phi
$$
induced by the morphism $\g^*_{\CP} \otimes \omega_C\to  ( \g^*\inv G)_{\omega_C}$.
One of the main goals of the Geometric Langlands program is the quantization
of the Hitchin fibers.
Given a parameter $\flat\in \mathbb B_G(C)$, with Hitchin fiber 
$\CL_\flat=\CH^{-1}(\flat) \subset T^* \Bun_G(C)$,
we seek a $\cD$-module on $\Bun_G(C)$ whose ``support" is equal to $\CL_\flat$.
In physical terms, we would like to understand the structure of the A-brane wrapping $\CL_\flat$.

In the toy case when $C$ is a cuspidal elliptic curve, the moduli of semistable
$G$-bundles (equivalently, bundles whose pullback to the normalization $\mathbb P^1$
are trivializable) reduces to $\g/G$, and the
Hitchin fibration reduces to a form of the moment map. 
To simplify the discussion, let us use the Killing form to identify $\g$ and its dual $\g^*$,
and fix a Cartan sublagebra $\fh\subset \g$.
Then under the resulting induced identifications
$$
T^*(\g/G) \simeq\{(x, \xi)\in \g\times \g | [x, \xi] = 0\}/G
\qquad
\g^*\inv  G \simeq \fh\inv W
$$
the restriction of the Hitchin fibration takes the form
$$
\CH: T^*(\g/G)\to \fh\inv W
\qquad
\CH(x, \xi) = \ol \xi
$$
where $\ol \xi\in  \fh\inv W$ denotes the class of $\xi\in \g$.
As an application of our main result,
we show that
for a regular parameter $\lambda\in \fh\inv W$,
the semistable Hitchin fiber $\CL_\lambda =\CH^{-1}(\lambda)\subset T^*(\g/G)$ is the balanced brane
corresponding to the Fourier transform of the Springer sheaf $S_\CN$.
Motions of the regular parameter provide a braid group action
that descends to the usual Weyl group action of Springer theory.


\subsection{Outline} Here is a further outline of the contents 
and arguments of the paper.

\medskip

In Section~\ref{sec sheaves}, we develop for our purposes some foundational material on the
constructible derived category $D_c(X)$ of a real analytic manifold $X$.
We explain how to form the standard differential graded (dg) category $Sh_c(X)$
with cohomology category $D_c(X)$. We then collect background material
on the Fourier transform for constructible sheaves. 
To formulate our main theorem, we need to take certain limits of constructible sheaves.
Let $V$ be a real finite-dimensional 
vector space with dilation action $\alpha^t:V\to V$, for scalars $t\in \R^+$.
For $\CF$ a constructible sheaf on $V$, we construct a sheaf $\Upsilon(\CF)$
that formalizes the notion of the conic limit 
$``\lim_{t\to 0} \alpha^t_*(\CF)"$. We describe its sections over certain cones in $V$
as needed in our main theorem.

\medskip

In Section~\ref{sec fukaya}, we review the Fukaya category of the cotangent
bundle $T^*X$ of a compact real analytic manifold $X$ following~\cite{NZ, N}. We review the microlocalization quasi-equivalence
$$
\xymatrix{
\mu_X:Sh_c(X)\ar[r]^{\sim} & F(T^*X)
}$$
from constructible sheaves to the triangulated envelope of the Fukaya category.
We will not need the full import of this result, only that $\mu_X$ is a quasi-embedding.
But we will need the main ingredient in the proof: the invariance of Floer calculations under non-characteristic motions
explained in \cite[Section 3.7]{N}. This is the fact that Hamiltonian
isotopies of noncompact branes during which no critical event occurs
near infinity lead to quasi-isomorphic calculations.

\medskip

Section~\ref{sec fourier for branes} contains the main technical work
of the paper.
Given a real finite-dimensional vector space $V$ with dual $V^*$,
we study noncompact Lagrangian branes $L\hra V\times V^*$.
We describe how such an object $L$ gives rise to constructible sheaves 
$\pi_V(L)$ and $\pi_{V^*}(L)$ on the respective factors $V$ and $V^*$.
We isolate a class of Lagrangian branes, which we call balanced branes,
to which our main theorem applies. 
Recall that for $\CF$ a constructible sheaf on a vector space,
 $\Upsilon(\CF)$ denotes its conic limit $``\lim_{t\to 0} \alpha^t_*(\CF)"$.
Our main result, Theorem~\ref{main thm}, is the following.

\begin{thm}
Let $L\hra V_1\times V_2$ be a balanced brane. The Fourier transform
and its inverse
exchange the conic limits of the corresponding constructible sheaves 
$$
(\Upsilon(\pi_V(L)))^\w \simeq \Upsilon(\pi_{V^*}(L))
\qquad
(\Upsilon(\pi_{V^*}(L))^\vee \simeq \Upsilon(\pi_V(L)).
$$
\end{thm}

In its most succinct form, the proof of the theorem takes the following shape.
One observes that a pair of constructible sheaves are exchanged by the Fourier 
transform if  and only if certain
 invariants of the sheaves are exchanged. Namely, the sections of the Fourier transform 
 of a sheaf over an
 open convex cone are equal to the sections of the original sheaf within the closed polar cone.
 One shows that for a pair of sheaves
coming from a single brane, these invariants
can be understood as a single invariant within the Fukaya category.
Namely, for a given pair of an open convex cone and its closed polar cone, 
there is a single brane such that pairing with it
realizes both spaces of sections. The argument involves a delicate application of the 
invariance of Floer calculations under
non-characteristic motions.
To help the reader, we include a version of the argument
in the case $\dim V = 1$ where one can gently acclimate to some of the intricacies
involved.

\medskip

Finally, in Section~\ref{sec springer}, we apply the preceding theory to the adjoint
quotient $\g/G$. We first give a brief synopsis of Springer theory in the language
of perverse sheaves.
We then review the standard formalism for working with the cotangent 
bundle of a stack such as $\g/G$. We then introduce the Hitchin fibration
in the case of the cotangent bundle $T^*(g/G)$. Finally, we show that
its regular fibers define balanced branes in $T^*(g/G)$. We use our main result
to deduce that these branes give rise to the perverse sheaves of Springer theory.



\subsection{Acknowledgements}
It is a pleasure to thank Eric Zaslow for his generous interest and encouragement.
I am grateful to Pierre Schapira for helpful remarks on the exposition.
I am also grateful to an anonymous referee for careful, detailed comments on the proofs and the foundations they rely upon.
Finally, I would like to acknowledge the support of NSF grant DMS-0600909 and a Sloan Research
Fellowship.



\section{Constructible sheaves}\label{sec sheaves}

In this section, we first review background material on the constructible derived category,
then recall a differential graded model of it. Finally, we review the Fourier-Sato transform.


\subsection{Derived category}
In this section, we briefly recall the construction of the constructible derived category of
a real analytic manifold.
For a comprehensive treatment of this topic, the reader could consult
the book of Kashiwara-Schapira~\cite{KS}. 

\medskip
Let $X$ be a topological space. Let $\Top(X)$ be the category whose objects
are open sets $U\hra X$, and morphisms are 
inclusions $U_0 \hookrightarrow U_1$ of open sets:
$$
\hom_{\Top(X)} (U_0, U_1)
=
\left\{
\begin{array}{cl}
pt & \mbox{ when $U_0 \hra U_1$,} \\
\emptyset & \mbox{ when $U_0 \not\hra U_1$.}
\end{array}
\right.
$$

\medskip

Let $\Vect$ be the abelian category of complex vector spaces.

\medskip

The derived category of sheaves of complex vector spaces on $X$ is 
traditionally defined
via the following sequence of constructions:

\medskip
\noindent
1. {\em Presheaves}. Presheaves on $X$ are functors $\CF: \Top(X)^\circ\to \Vect$
where $\Top(X)^\circ$ denotes the opposite category. Given an open set $U\hra X$,
one writes $\CF(U)$ for the sections of $\CF$ over $U$, and given an inclusion $U_0 \hra U_1$
of open sets, one writes $\rho^{U_1}_{U_0}: \CF(U_1)\to \CF(U_0)$ for the
corresponding restriction map.

\medskip
\noindent
2. {\em Sheaves}. Sheaves on $X$ are presheaves $\CF: \Top(X)^\circ\to \Vect$
which are locally determined in the following sense.
For any open set $U\hra X$,
and covering $\fU=\{U_i\}$ of $U$ by open subsets $U_i\hra U$, there is a complex of vector spaces
$$
\xymatrix{
0 \ar[r] & \CF(U) \ar[r]^-{\delta} & \prod_i \CF(U_i) \ar[r]^-{\delta_0} & \prod_{i,j} \CF(U_i\cap U_j),
}
$$ 
where $\delta = \prod_i \rho^U_{U_i}$ and 
$\delta_0 = \prod_{i,j} \left (\rho^{U_i}_{U_i\cap U_j} - \rho^{U_j}_{U_i\cap U_j}\right )$.
A sheaf is a presheaf for which $\ker(\delta) = \ker(\delta_0)/\on{im}(\delta)=0$ for all
open sets 
and coverings of open sets.
%

\medskip

Sheaves on $X$ form an abelian category and thus one can continue with
the following sequence of general homological constructions:

\medskip
\noindent
3. {\em Complexes}. Let $C(X)$ be the abelian category of complexes of sheaves on $X$
with morphisms the degree zero chain maps.
Given a complex of sheaves $\CF$, one writes $H(\CF)$ for
the (graded) cohomology sheaf of $\CF$.

\medskip
\noindent
4. {\em Homotopy category}. Let $K(X)$ be the homotopy category of sheaves on $X$
with objects complexes of sheaves and morphisms homotopy classes of maps.
This is a triangulated category whose distinguished triangles are isomorphic to 
the standard mapping cones.

\medskip
\noindent
5. {\em Derived category}. The derived category $D(X)$ of sheaves on $X$
is defined to be the localization of $K(X)$ with respect to homotopy classes of quasi-isomorphisms (maps
inducing isomorphisms on cohomology).
Acyclic objects form a null system in $K(X)$, and thus $D(X)$ inherits the structure of 
triangulated category.

\medskip

With the derived category $D(X)$ in hand, one can define many variants by imposing
topological and homological conditions on objects.

\medskip
\noindent
6. {\em Bounded derived category}.
The bounded derived
category $D^\flat(X)$ is defined to be the full subcategory of $D(X)$ of bounded complexes.

Two standard equivalent descriptions are worth keeping in mind: first,
there is the more flexible description of $D^\flat(X)$ as
the full subcategory of $D(X)$ of complexes with bounded cohomology; second, 
there is the computationally useful description of
$D^\flat(X)$ as the homotopy category
of complexes
of injective sheaves with bounded cohomology.

\medskip
\noindent
7. {\em Constructibility}.  Assume $X$ is a real analytic manifold.
Fix an analytic-geometric category
$\CC$ in the sense of~\cite{vdDM}. For example, one could take $\CC(X)$ to be the
subanalytic subsets of $X$ as described in~\cite{BieMi}.

Let $\CS=\{S_\alpha\}$ be a Whitney stratification of $X$ by $\CC$-submanifolds
$i_\alpha: S_\alpha\hra X$.
An object $\CF$ of $D(X)$ is said to be $\CS$-constructible 
if the restrictions
$
i_\alpha^*H(\CF)
$
of its cohomology sheaf
to the strata of $\CS$ are finite-rank and locally constant.

The $\CS$-constructible derived category $D_\CS(X)$ is the full subcategory of $D(X)$ of $\CS$-constructible objects.
The constructible derived category $D_c(X)$ is the full subcategory of $D(X)$ of objects which are $\CS$-constructible for some Whitney stratification $\CS$.

Note that if the stratification $\CS$ is finite (for example, if $X$ is compact), then the finite-rank condition implies that all $\CS$-constructible objects have bounded cohomology. In other words,
within $D(X)$, 
every object of $D_\CS(X)$ is isomorphic to an object of $D^\flat(X)$.


\subsection{Differential graded category}
The derived category $D(X)$ is naturally the cohomology category of a differential
graded (dg) category $Sh(X)$. To define it, we will return to the sequence of 
homological constructions 
listed above and perform some modest changes. 
Two principles guide 
such definitions: (1) structures (such as morphisms and higher exts) should be defined at the level
of complexes not their cohomologies; 
and (2) properties (such as constructibility) should be imposed at the level of cohomologies rather than
complexes. The first principle ensures we will not lose important information,
while the second ensures we will have sufficient flexibility.
As an example of the latter, we prefer the realization of the bounded derived
category $D^\flat(X)$ as the full subcategory of $D(X)$ of complexes with bounded cohomologies
rather than of strictly bounded complexes.

The reader could consult~\cite{Drinfeld, KellerICM} 
for background on dg categories, in particular,
a discussion of the construction of dg quotients.

\medskip

Recall that sheaves on $X$ form an abelian category. The following sequence of
homological constructions can be performed on any abelian category:

\medskip
\noindent
1. {\em Dg category of complexes}. Let $C_{dg}(X)$ be the dg category
with objects complexes of sheaves and morphisms the usual complexes of maps between complexes.
In particular, the degree zero cycles in such a morphism complex are the usual degree zero
chain maps which are the morphisms of the ordinary category $C(X)$.

\medskip
\noindent
2. {\em Dg derived category}. The dg derived category $Sh(X)$
is defined to be the dg quotient of $C_{dg}(X)$ by the full subcategory of acyclic objects.
This is a triangulated dg category whose
cohomology category $H(Sh(X))$ is canonically equivalent (as a triangulated category) to the usual derived category $D(X)$.

\medskip

One can cut out full triangulated dg subcategories of $Sh(X)$ by specifying
full triangulated subcategories of its cohomology category $H(Sh(X))\simeq D(X)$. 

\medskip
\noindent
3. {\em Bounded dg derived category}.
The bounded dg
derived category $Sh^\flat(X)$ is defined to be the full dg subcategory of $Sh(X)$ 
of objects projecting to $D^\flat(X)$.

\medskip
\noindent
4.  {\em Constructibility}. 
 Assume $X$ is a real analytic manifold, and 
fix an analytic-geometric category $\CC$.
The constructible dg derived category $Sh_{c}(X)$ is the full dg subcategory of $Sh(X)$
of objects projecting
to $D_c(X)$. 
For a Whitney stratification $\CS$ of $X$, the $\CS$-constructible dg derived category 
$Sh_{\CS}(X)$ is the full dg subcategory of $Sh(X)$
of objects projecting
to $D_\CS(X)$.

\medskip

The formalism of Grothendieck's six (derived) operations $f^*, f_*, f_!, f^!, \intHom, \otimes$
can be lifted to the constructible dg derived category $Sh_c(X)$ 
(see for example~\cite{Drinfeld} for a general discussion of deriving
functors in the dg setting). In our case, one concrete approach
is to recognize that the natural map $C_{dg, c}(\mathfrak{Inj}(X)) \to Sh_c(X)$ 
from the dg category $C_{dg, c}(\mathfrak{Inj}(X))$ of complexes of injective sheaves
with constructible cohomology is a quasi-equivalence.
With this in hand, one can define derived functors by evaluating their
naive versions on $C_{dg, c}(\mathfrak{Inj}(X))$.
Since we will only consider derived functors, we will denote them by the above unadorned symbols.

\medskip

Throughout the remainder of this paper, we
fix an analytic-geometric category $\CC$. All subsets will be $\CC$-subsets
unless otherwise stated.


\subsection{Fourier transform for sheaves}

We recall here the Fourier-Sato transform following~\cite[Section 3.7]{KS}.
We will describe the general parameterized version over a base manfold $X$,
though our application
will involve only the absolute version over $X=pt$.

When working with constructible complexes 
on a noncompact manifold $E$,
it is often technically convenient to fix a relative compactification $i:E\hra \ol E$
and work with constructible complexes on $\ol E$.
In the case of a vector bundle $E\to X$,
we will always take the relative spherical compactification
$$
\ol E =\left( (E\times  \R^{\geq 0}) \setminus (X \times \{0\})\right)/ \R_+.
$$ 
By a constructible complex on $E$, we will mean a complex $\CF$ on $E$ such that
its extension $i_*\CF$ to $\ol E$ is constructible.
Throughout this section, we will abuse notation and write $Sh_c(E)$ for the full dg subcategory of such 
complexes.
Note in particular that for the natural dilation $\R^+$-action on $E$, any $\R^+$-equivariant 
complex on $E$
that is constructible in the usual sense is also constructible in the above extended sense.

\subsubsection{Definition of Fourier transform}
Let $\pi_1:E\to X$ be a real finite-rank vector bundle, and let $\pi_2:E^*\to X$
be the dual vector bundle.
For $e\in E$ and $e^* \in E^*$, let $\langle e, e^*\rangle\in \R$ denote the natural
pairing.

Consider the Cartesian diagram of vector bundles 
$$
\xymatrix{
 & \ar[ld]_-{p_1} E \times_X E^*  \ar[rd]^-{p_2}& \\
 E  \ar[rd]_-{\pi_1} & & E^*  \ar[ld]^-{\pi_2} \\
 & X &
}
$$

Consider the closed subset of non-positive pairs
$$
\xymatrix{
\kappa: K=\left\{
(e, e^*)\in E\times_X E^* | \langle e,  e^*\rangle \leq 0 \right\}
\ar@{^(->}[r]  & E\times_X  E^*
}
$$
and define the integral kernel 
$$
\CK=\kappa_!\C_K\in Sh_c(E \times_X E^*).
$$
Define integral transforms by the formulas
$$
\Phi_K: Sh_c(E) \to Sh_c(E^*)
\qquad
\Psi_K: Sh_c(E^*) \to Sh_c(E)
$$
$$
\CF^\w=\Phi_\CK(\CF) =  p_{2!}(\CK\otimes p_1^*\CF) = p_{2!}\kappa_!\kappa^*p_1^*\CF
$$
$$
\CG^\vee=\Psi_\CK(\CG) = p_{1*}\intHom(\CK, p_2^!\CG)
= p_{1*}\kappa_*\kappa^!p_2^{!}\CG
$$
By usual formalism (independent of the particular kernel $\CK$), 
the functors form an adjoint pair $(\Phi_\CK, \Psi_\CK)$.

\subsubsection{Conic sheaves}
Let $\R^+\subset \R^\times$ denote the multiplicative group of positive real numbers.
For a compact manifold $X$ with an $\R^+$-action, let $\alpha:\R^+\times X\to X$ be the action map,
$X/\R^+$ the quotient stack, 
and $p:X\to X/\R^+$ the natural projection.

\medskip

We refer to $Sh_c(X/\R^+)$ as the $\R^+$-equivariant constructible dg derived category.
Since $\R^+$ is contractible,
$\R^+$-equivariance is a property not a structure.
Thus pullback (or in other words, the forgetful functor) $p^*:Sh_c(X/\R^+)\to Sh_c(X)$ identifies $Sh_c(X/\R^+)$ with the full subcategory of $Sh_c(X)$
of $\R^+$-equivariant objects.
We will often refer to $\R^+$-equivariant objects as conic objects.

We have the induction functors
$$
\gamma=p_![-1]: Sh_c(X) \to Sh_c(X/\R^+)
\qquad
\Gamma=p_*:Sh_c(X) \to Sh_c(X/\R^+)
$$
By usual formalism, $\gamma$ is a left adjoint to the forgetful
functor $p^*$, and $\Gamma$ is right adjoint. (For the former, we have used
that $\R^+$ is smooth, one dimensional, and has a canonical ``positive" orientation.)

\medskip

We will next describe an alternative way to construct conic objects by
 taking the limit at zero of ordinary ones.
To informally describe this, let $\alpha^t:X\to X$ denote the action $\alpha:\R^+\times X\to X$
evaluated at $t\in \R^+$.
Then given an object $\CF\in Sh_c(X)$,
we would like to make sense of the limit $\lim_{t\to 0} \alpha^t_*\CF$ as an object
of $Sh_c(X)$.

To make this precise, let $\R^{\geq 0}\subset \R$ denote the non-negative real numbers.
Consider the diagram
$$
\xymatrix{
X & \ar[l]_-{\alpha} X\times \R^{+} \ar@{^(->}[r]^-{i_+}  & X\times \R^{\geq 0} & \ar@{_(->}[l]_-{j_0} 
X \times\{0\}\simeq X
}
$$
where $\alpha$ is the action map, and $i_+$ and $j_0$ are the obvious inclusions.
There is a natural $\R^+$-action on the above diagram where $\R^+$ acts trivially
on the $X$ at the left, diagonally on $ X\times \R^{+}$ and  $X\times \R^{\geq 0}$ and by the action $\alpha$
on the $X$ at the right.
Thus we can pass to the quotient diagram
$$
\xymatrix{
X
& \ar[l]_-{\alpha} (X\times \R^{+})/\R^+ \ar@{^(->}[r]^-{i_+}  & (X\times \R^{\geq 0})/\R^+ 
& \ar@{_(->}[l]_-{j_0} X\times\{0\}/\R^+ \simeq X/\R^+
}
$$

Now we define the limit functor to be the composition
$$
\Upsilon= j_0^*i_{+*}\alpha^*:Sh_c(X) \to Sh_c(X/\R^+)
$$

Let us relate this construction to the usual nearby cycles functor.
Suppose $X$ is a complex manifold, and the $\R^+$-action on $X$ extends to a $\C^\times$-action
$\alpha:X\times\C^\times \to X$.
Then we can use the action map $\alpha$ together with the nearby cycles 
$$
R\psi:Sh_c(X\times \C^\times) \to Sh_c(X),
$$ 
with respect to the canonical projection
$X\times \C \to \C$,
to define 
a functor
$$
R\psi\circ \alpha^*[1]:Sh_c(X) \to Sh_c(X).
$$
In fact, since the action map $\alpha$ is naturally $\C^\times$-equivariant
for the trivial action on $X$ and diagonal action on $X\times\C^\times$,
and the nearby cycles $R\psi$ makes sense in the equivariant setting,
 the functor $R\psi\circ \alpha^*[1]$ canonically lifts to 
a functor $Sh_c(X)\to Sh_c(X/\C^\times)$.

With the usual conventions (so that $R\psi$ preserves perverse sheaves),
the limit construction and nearby cycles construction are equivalent

$$
\Upsilon\simeq R\psi\circ \alpha^*[1].
$$


\subsubsection{Properties of Fourier transform}
Consider the natural $\R^+$-action on the vector bundles $E$, $E^*$.
Since $K\subset E\times_X E^*$ is $\R^+\times \R^+$-invariant,
and so $\CK=\kappa_!\C_K$ is $\R^+\times \R^+$-equivariant,
the Fourier transforms $\Phi_\CK$, $\Psi_\CK$ land in the full subcategories
of conic objects.

Here are two main properties of the Fourier transforms:

\smallskip

(1) \cite[Theorem 3.7.9]{KS}
The restrictions of the Fourier transforms $\Phi_\CK$, $\Psi_\CK$ 
to the full subcategories $Sh_{c}(E/\R^+)$, $Sh_{c}(E^*/\R^+)$ 
of conic objects
are inverse equivalences.

\smallskip

(2) \cite[Proposition 10.3.18]{KS}
 Suppose $X$ is a complex manifold and $E$ is a complex vector bundle.
Then the
restrictions of the shifted Fourier transforms $\Phi_\CK[\dim_\C E]$, $\Psi_\CK[-\dim_\C E]$ 
to the full subcategories $Perv(E/\R^+)$, $Perv(E^*/\R^+)$ 
of conic perverse sheaves
are inverse equivalences.

\smallskip

Finally, the Fourier transforms are compatible with equivariantization:
by standard identities among functors, there are canonical quasi-isomorphisms
$$
\Phi_\CK(\gamma(\CF)) \simeq \Phi_\CK(\CF)
\qquad
\Psi_\CK(\Gamma(\CF)) \simeq \Psi_\CK(\CF)
$$


\subsubsection{Characterizing calculations}

It will be useful to recall the explicit result of evaluating the Fourier transform
on basic objects. In particular, in the proof of our main theorem,
we will need the following elementary calculation.

\medskip

Recall that small open convex subsets $b:B\hra E$ generate the topology of $E$.
To identify an object $\CF\in Sh_c(E)$, it suffices to know its sections
$\Gamma(B, \CF) \simeq \hom_{Sh_c(E)}(b_!\C_B, \CF)$
over small open convex subsets, along with the restriction maps among the sections.

By a cone $u:U\hra E$, we will mean an $\R^+$-invariant subset
(or in other words, the inverse image of a subset $U/\R^+\hra E/\R^+$).
To identify an object $\CF\in Sh_c(E/\R^+)$, it suffices to know its sections
$\Gamma(U, \CF) \simeq \hom_{Sh_c(E)}(u_!\C_U, \CF)$
over open convex cones, along with the restriction maps among the sections.

\medskip

We have the following useful description (see~\cite[Proposition 3.7.12]{KS}) 
of the sections of the Fourier transform
over convex open subsets.

For $u:U\hra E^*$ a subset, define the closed polar cone 
$$
v: U^\circ =\left\{e \in E\times_X \pi_2(U)  |
\langle e, e^* \rangle \geq 0
\mbox{ for all } e^* \in  \{\pi_1(e)\}\times_X U \right\} \hra E.
$$
Now if $u:U\hra E^*$ is an open convex cone 
with polar cone $v:U^\circ\hra E$ then 
(see~\cite[Proposition 3.7.12]{KS}) there is a quasi-isomorphism
$$
\pi_{2*}u^*\Phi_\CK(\CF)\simeq
\pi_{1*}v^{!}\CF,
\quad \mbox{ for all $\CF\in Sh_{c}(E/\R^+)$.}
$$
Note that $U$ is open so $u^*\simeq u^!$, and $U^\circ$ is closed so $u_* \simeq u_!$.
By adjunction, we can interpret this as a quasi-isomorphism of morphism complexes
$$
\hom_{Sh_c(E^*)}(u_!\C_U, \Phi_\CK(\CF)) \simeq
\hom_{Sh_c(E)}(v_*\C_{U^\circ}, \CF)
$$
In particular, one can deduce the above identification from the calculation
$
\Phi_\CK(v_*\C_{U^\circ}) \simeq u_!\C_{U}
$
(see~\cite[Lemma 3.7.10]{KS}).

Furthermore, an inclusion of open convex cones $U_0 \hra U_1 \hra E^*$,
leads to an inclusion morphism $u_{0!}\C_{U_0} \to u_{1!}\C_{U_1}$,
and an inclusion morphism $v_{1!}\omega_{U_1^\circ} \to v_{0!}\omega_{U_0^\circ}$.  
In turn, for all $\CF\in Sh_{c}(E/\R^+)$, these induce
a commutative (at the level of cohomology) square 
$$
\xymatrix{
\hom_{Sh_c(E^*)}(u_{1!}\C_{U_1}, \Phi_\CK(\CF)) \ar[r]^-\sim \ar[d] & \ar[d]
\hom_{Sh_c(E)}(v_{1*}\C_{U_1^\circ}, \CF)\\
\hom_{Sh_c(E^*)}(u_{0!}\C_{U_0}, \Phi_\CK(\CF)) \ar[r]^-\sim &
\hom_{Sh_c(E)}(v_{0*}\C_{U_0^\circ}, \CF)
}
$$

In conclusion, we see that for $\CF\in Sh_{c}(E/\R^+)$,
the Fourier transform $\Phi_\CK(\CF)$ is characterized
by the above calculations. 

\medskip

Before moving on, it will be convenient to go one step further
and note the following. Suppose $u:U\hra E^*$ is an open convex cone but does not
contain any fiber of $E^*$. Then the interior $\interiorsymbol  (v):\interiorsymbol  (U^\circ)\hra E$
of the closed polar cone $v:U^\circ\hra E$ is an open convex submanfold.
We have a quasi-isomorphism $v_*\C_{U^\circ} \simeq \interiorsymbol  (v)_*\C_{\interiorsymbol  (U^\circ)}$,
and thus a quasi-isomorphism
$$
\hom_{Sh_c(E^*)}(u_!\C_U, \Phi_\CK(\CF)) \simeq
\hom_{Sh_c(E)}(\interiorsymbol  (v)_*\C_{\interiorsymbol  (U^\circ)}, \CF),
$$
for all $\CF\in Sh_{c}(E/\R^+)$. 
(Also, when $u:U\hra E^*$ equals all of $E$, the closed polar cone $v:U^\circ\hra E$
is the zero section $X\hra E$.) On both sides of the above identification,
we have a pairing with the standard or costandard sheaf of a submanifold.
It is this form of the calculation that we will use to characterize the Fourier transform.


\subsubsection{Sections of limit}
For future reference, we collect here the calculation of the sections
of the limit $\Upsilon(\CF)\in Sh_c(\R^n/\R^+)$ of an object $\CF\in Sh_c(\R^n)$
with respect to the dilation $\R^+$-action.

For $t\in \R^+$, let $\alpha^t:\R^n \to \R^n$ denote the dilation $\R^+$-action.
For each $\varepsilon\in \R$, consider the open quadrant 
$$
q^\varepsilon: Q^\varepsilon = \{ (x_1, \ldots, x_n)\in \R^n | x_i > \varepsilon\} \hra \R^n.
$$
We will write $q:Q\hra \R^n$ for the open cone $q^0:Q^0\hra \R^n$. 
Note that linear transformations of $Q$,
together with $\R^n$ itself, 
form a basis for the conic topology of $\R^n$.

\begin{lem}\label{lem sections calc}
For any object $\CF\in Sh_c(\R^n)$, and any $\varepsilon\in \R^+$, there exists 
$\delta(\vareps)\in \R^+$ 
such that for any $\delta\in (0,\delta(\vareps))$, there are canonical identifications 
$$
\hom_{Sh_c(E)}(q_!\C_Q , \Upsilon(\CF))
\simeq
\hom_{Sh_c(E)}(q^\varepsilon_!\C_{Q^\epsilon} , \alpha^\delta_*(\CF))
$$
$$
\hom_{Sh_c(E)}(q_*\C_Q , \Upsilon(\CF))
\simeq
\hom_{Sh_c(E)}(q^{-\varepsilon}_*\C_{Q^{-\epsilon}}, \alpha^\delta_*(\CF))
$$
\end{lem}

\begin{proof}

Recall that $i:\R^n\hra \ol \R^n$  denotes the spherical compactification,
and by assumption, there is a stratification $\CS$ of 
$\ol \R^n$ such that $i_*\CF$
is $\CS$-constructible.

Consider the $\R^+$-conic stratification $C(\CS)$ of $\R^n \times \R^+$ obtained
by taking the cone over $\CS$ together with the origin.
Observe that the object $i_{+*}\alpha^*\CF$ 
appearing in the definition of $\Upsilon(\CF)$ is $C(\CS)$-constructible.

Fix any $\vareps_0\in \R^+$.
Let $\CQ^{\pm\vareps_0}$ be the stratification of $\R^n$ given by the facets of $Q^{\pm\vareps_0}$,
and let 
$\CQ^{\pm\vareps_0}\times \{\R^{+}, \{0\}\}$ be the product stratification of $\R^n \times \R^{\geq 0}$.  

Let $\CR^{\pm\vareps_0}$ be a stratification of $\R^n\times\R^{\geq 0}$, that extends to $\ol \R^n\times \R^{\geq 0}$,
and refines $C(\CS)$ and 
$\CQ^{\pm\vareps_0}\times \{\R^{+}, \{0\}\}$.
The restriction of the projection $\R^n\times \R^{\geq 0} \to \R^{\geq 0}$ to the strata of $\CR^{\pm\vareps_0}$
will have discrete critical values. In particular, there exists $\delta(\vareps_0)\in \R^+$
such that there are no critical values in the interval $(0,\delta(\vareps_0))$.
Thus by the Thom Isotopy Lemma, there is a stratified homeomorphism of
$\ol \R^n \times (0,\delta(\vareps_0))$ taking the restriction of the stratification $\CR^{\pm\vareps_0}$
to a product stratification.

Observe that the above constructions behave well with respect to dilation.
For $t\in \R^+$ and $\vareps = t\vareps_0$, we can take the dilation
of the above stratifications and homeomorphism for $\vareps_0$. In particular, 
for $\delta(\vareps) = t\delta(\vareps_0)$,
 there is a stratified homeomorphism of
$\ol \R^n \times (0,\delta(\vareps))$ taking the restriction of the stratification $\CR^{\pm\vareps}$
to a product stratification.

With the above system of open neighborhoods in hand,
standard sheaf identities establish the identifications for any $\delta\in (0,\delta(\vareps))$.
\end{proof}


\section{Microlocal branes}\label{sec fukaya}
In this section, we review some basic aspects of the Fukaya category of a cotangent bundle $T^*X$
of a compact real analytic manifold $X$ following~\cite{NZ, N}.
In particular, we recall the quasi-equivalence between constructible sheaves on $X$
and the triangulated envelope of the Fukaya category of $T^*X$.
For the foundations of the Fukaya category, we refer the reader to the fundamental
work of Fukaya-Oh-Ohta-Ono~\cite{FOOO} and Seidel~\cite{Seidel}.


\subsection{Preliminaries}\label{sect prels}

In what follows, we work with a fixed compact real analytic manifold $X$
with cotangent bundle $\pi:T^*X\to X$. We often denote points of $T^*X$ by pairs
$(x,\xi)$ where $x\in X$ and $\xi\in T^*_x X$.
The material of this section is a condensed version
of the discussion of~\cite{NZ}. 

\medskip

Let $\theta\in\Omega^1(T^*X)$ denote the canonical one-form $\theta(v) =\xi(\pi_*v)$, for
$v\in T_{(x,\xi)}(T^*X)$,  and let 
$\omega =d\theta\in\Omega^2(T^*X)$ denote the 
canonical symplectic structure. 
For a fixed Riemannian metric on $X$, let $|\xi|:T^*X\to \R$ denote
the corresponding fiberwise linear length function.


\subsubsection{Compactification} To better control noncompact Lagrangians
in $T^*X$, it is useful to work with 
the cospherical compactification $\ol\pi:\ol T^*X\to X$ of the projection
$\pi:T^*X\to X$ obtained by attaching the cosphere bundle at infinity $\pi^\infty:T^\infty X\to X$.

Concretely, we can realize the compactification $\ol T^*X$ as the quotient
$$
\ol T^*X = \left((T^* X \times \R_{\geq 0})\setminus (X \times \{0\})\right)/ \R_+$$
where
$\R_{+}$ acts 
by dilations on both factors.
The canonical inclusion $T^*X\hookrightarrow \ol T^* X$ sends
a covector $\xi$ to the class of $[\xi,1]$.
The boundary at infinity 
$T^\infty X=\ol T^*X\setminus T^*X
$ 
consists of classes of the form $[\xi, 0]$ with $\xi$ a non-zero covector.
Given a Riemannian metric on $X$, one can identify
$\ol T^*X$ with the closed unit disk bundle $D^*X$,
and $T^\infty X$ with the unit cosphere bundle  $S^* X$,
via the map
$$[\xi,r]\mapsto (\hat \xi, \hat r),
\mbox{ where }|\hat\xi|^2 + \hat r^2=1.
$$

The boundary at infinity $T^\infty X$ carries a canonical contact 
distribution $\kappa \subset T (T^\infty X)$ with a well-defined notion of positive normal direction.
Given a Riemannian metric on $X$, under the induced identification of $T^\infty X$
with the unit cosphere bundle $S^* X$, the distribution $\kappa$ is
the kernel of the restriction of $\theta$.


\subsubsection{Conical almost complex structure}
\label{warp}
To better control holomorphic disks in $T^*X$, it is useful to work with 
an almost complex structure $J_{con}\in\on{End}(T(T^*X))$ 
which near infinity is invariant under dilations.

A fixed Riemannian metric on $X$ provides a canonical splitting
$
T(T^*X) \simeq T_{b} \oplus T_{f},
$
where $T_b$ denotes the horizontal base directions and $T_f$ the 
vertical fiber directions,
along with a canonical isomorphism 
$
j_0:T_b {\to} T_f
$
of vector bundles over $T^*X$.
We refer to the resulting almost complex structure
$$
J_{Sas}=
\left(
\begin{matrix}
0 & j_0^{-1} \\
-j_0 & 0
\end{matrix}
\right)
\in\on{End}(T_b \oplus T_f).
$$
as the Sasaki almost complex structure,
since by construction, the Sasaki metric on $T^*X$
is given by $g_{Sas}(v,v) = \omega(u, J_{Sas} v)$.

Fix positive constants $r_0, r_1>0$,
a bump function $b: \R\to \R$ such that $b(r) = 0$ for $r <r_0$, and $b(r) = 1$, for $r>r_1$,
and
set 
$w(x,\xi) = |\xi|^{b(|\xi|)},
$ where as usual
$|\xi|$ denotes the length of a covector with respect to the original
metric on $X$. 
We refer to the compatible almost complex structure 
$$
J_w=
\left(
\begin{matrix}
0 & w^{-1} j_0^{-1} \\
-w j_0 & 0
\end{matrix}
\right)
\in\on{End}(T_b \oplus T_f)
$$
as a(n asymptotically) conical almost complex structure
since near infinity $J_{con}$ is invariant under dilations.
The corresponding metric $g_{con}(u,v) = \omega(v, J_{con} v)$ presents
$T^*X$ near infinity as a metric cone over the unit {cosphere} bundle $S^*X$
equipped with the Sasaki metric.

One can view the conical metric $g_{con}$ as being compatible with the compactification $\ol T^*X$
in the sense that near infinity
it treats base and angular fiber directions on equal footing. Near infinity the metrics
on the level sets of $|\xi|$ are given by scaling the Sasaki metric
on the unit cophere bundle by the factor $|\xi|^{1/2}$.


\subsection{Brane structures}\label{sec branes}
By a Lagrangian $j:L\hookrightarrow T^*X$, we mean a closed (but not necessarily
compact) half-dimensional submanifold such that $TL$ is isotropic for the symplectic form $\omega$.
One says that $L$ is exact if the pullback of the one-form
$j^*\theta$ is cohomologous to zero. 

\medskip

By a brane structure on a Lagrangian ${L} \hra{T}^*X$, we mean a three-tuple
$
(\CE, \tilde\alpha,\flat)
$
consisting of a flat (finite-dimensional) vector bundle $\CE\to L$,
along with a grading $\tilde\alpha:L\to \R$ (with respect to the canonical
bicanonical trivialization) and a relative pin structure $\flat$ (with respect to the background class
$\pi^*(w_2(X))$. To remind the interested reader,
we include below a short summary of what the latter two structures entail.


\subsubsection{Gradings}
The almost complex structure $J_{con}\in\on{End}(T(T^*X))$ provides a holomorphic canonical bundle
$
\kappa= (\wedge^{\dim X}T^{hol}(T^*X))^{-1}.
$
According to~\cite{NZ}, there is a canonical trivialization $\eta^2$ of
the bicanonical bundle $\kappa^{\otimes 2}$ (and a canonical trivialization of $\kappa$
itself
if $X$ is assumed oriented).
Consider the bundle of Lagrangian planes ${\mathcal Lag}_{T^*X}\to T^*X$,
and the squared phase map
$$
\alpha : {\mathcal Lag}_{T^*X} \rightarrow U(1)
$$
$$
\alpha(\mathcal L)= \eta(\wedge^{\dim X} \mathcal L)^2/ |\eta(\wedge^{\dim X} \mathcal L)|^2.
$$

For a Lagrangian $L\hra  T^*X$ and a point $x\in L$, we obtain a map $\alpha:L\to U(1)$
by setting $\alpha(x)=\alpha(T_xL).$
The Maslov class $\mu(L)\in H^1(L,\Z)$
is the obstruction class
$
\mu = \alpha^*(dt),
$ 
where $dt$ denotes the standard one-form on $U(1)$.
Thus $\alpha$ has a lift to a map $\widetilde\alpha:L\to \R$
if and only if $\mu = 0$, and choices of a lift form a torsor over the group $H^0(L,\Z)$.
Such a lift $\widetilde \alpha:L\to \R$ is called a {grading} of the Lagrangian $L\hra T^*X$.


\subsubsection{Relative pin structures}
Recall that the group $Pin^+(n)$ is the double cover of
$O(n)$ with center $\Z/2\Z \times \Z/2\Z.$
A pin structure on a Riemannian manifold $L$ is a lift
of the structure group of $TL$ to $Pin^+(n).$
The obstruction to a pin structure is the second
Stiefel-Whitney class $w_2(L)\in H^2(L,\Z/2\Z)$,
and choices of pin structures form a torsor over the group
$H^1(L,\Z/2\Z)$. 

A relative pin structure on a
submanifold $L \hookrightarrow M$ with background class $[w]\in H^2(M,\Z/2\Z)$
can be defined as follows.  Fix a {\v C}ech cocycle $w$ representing $[w]$,
and let $w|_L$ be its restriction to $L$. Then a pin structure on $L$ relative to $[w]$ can be
defined to be  
an $w|_L$-twisted pin structure on $TL$. Concretely, this can be represented by a $Pin^+(n)$-valued
\v Cech $1$-cochain on $L$ whose coboundary is $w|_L$. 
Such structures are canonically independent of the choice of \v Cech representatives.

For Lagrangians $L\hra T^*X$, we will always consider relative pin structures $\flat$ on $L$
with respect to the fixed background class $\pi^*(w_2(X))\in H^2(T^*X,\Z/2\Z)$.


\subsection{Fukaya category}\label{sect Fukaya category}
We recall here the construction of the Fukaya $A_\infty$-category 
of the cotangent bundle $T^*X$ of a compact real analytic manifold $X$. 
Our aim is not to review all of the details, but only those relevant to our later proofs.
For more details, the reader could consult~\cite{NZ} and the references therein.


\subsubsection{Objects}

An object of the Fukaya category of $T^*X$ is a four-tuple
$
(L, \CE, \tilde\alpha,\flat)
$
consisting of an exact (not necessarily compact) closed Lagrangian submanifold ${L} \hra{T}^*X$
equipped with a brane structure:
this includes a flat vector bundle $\CE\to L$,
along with a grading $\tilde\alpha:L\to \R$ 
(with respect to the canonical
bicanonical trivialization) and a relative pin structure $\flat$ (with respect to the background class
$\pi^*(w_2(X))$. 

To ensure reasonable behavior near infinity, we place two assumptions on the
Lagrangian $L$.
First, 
consider the compactification $\ol T^*X$ obtained by adding to $T^*X$ the cosphere bundle
at infinity $T^\infty X$. 
Then we fix an
analytic-geometric category $\CC$ once and for all, and
assume that
the closure $\ol L \hra  \ol T^*X$ is a $\CC$-subset. Along with other nice properties, 
this implies the following two key facts: 
\begin{enumerate}
\item The boundary at infinity 
$$
L^\infty =\ol L\cap T^\infty X
$$
is an isotropic subset of $T^\infty X$ with respect to the induced contact structure.
\item There is a real number $r>0$
such that the restriction of the length function 
$$|\xi| :L \cap \{|\xi| > r\}\to \R$$
has no critical points.
\end{enumerate}
As discussed below, the above properties guarantee we can make sense of ``intersections at infinity".

Second, to have a manageable theory of pseudoholomorphic maps with boundary on
such Lagrangians,
we also assume the existence of a perturbation $\psi$ that moves the initial
Lagrangian $L$ to a nearby Lagrangian tame (in the sense of~\cite{Sikorav})
with respect to the conical metric $g_{con}$. 
As explained in the Appendix of~\cite{N}, all such perturbations lead to equivalent calculations, though not necessarily
by the most direct comparison of equations.
For a recollection of the basic idea, see Section 3.5 below.
It is worth commenting that the above encompasses all restrictions placed on the objects of the Fukaya category studied herein.
In particular, though it is often convenient to think of an object as ``asymptotically conical", the proofs do not appeal to such an independent notion,
only the technical assumptions above.

We use the term Lagrangian brane to refer to objects of the Fukaya category.
When there is no chance for confusion,
we often write $L$ alone to signify the Lagrangian brane.


\subsubsection{Morphisms}

To define the morphisms between two branes, we must perturb Lagrangians
so that their intersections occur in some bounded domain.
To organize the perturbations, we recall the inductive
notion of a fringed set $R_{d+1}\subset \R^{d+1}_+$.
A fringed set $R_1\subset\R_+$ is any interval of the form $(0,r)$ for some $r>0$. 
A fringed set $R_{d+1}\subset\R_+^{d+1}$ is a subset satisfying the following:
\begin{enumerate}
\item $R_{d+1}$ is open in $\R^{d+1}_+$.
\item Under the projection $\pi:\R^{d+1}\to \R^d$ forgetting the last coordinate, 
the image $\pi(R_{d+1})$ is a fringed set.
\item If $(r_1,\ldots, r_d, r_{d+1})\in R_{d+1}$, then $(r_1,\ldots, r_d, r'_{d+1})\in R_{d+1}$
for $0<r'_{d+1}< r_{d+1}$.
\end{enumerate}

A Hamiltonian function $H:T^*X\to \R$ is said to be {controlled} if there is a real number
$r>0$ such that in the region $|\xi|>r$ we have $H(x,\xi)=|\xi|$. 
The corresponding Hamiltonian isotopy 
$\varphi_{H,t}:T^*X\to T^*X$ equals the normalized geodesic flow $\gamma_t$
in the region $|\xi|>r$.

As explained in~\cite{NZ},
given Lagrangians branes $L_0,\ldots, L_d\subset T^*X$, 
and controlled Hamiltonian functions $H_0,\ldots, H_d$,
we may choose a fringed set $R\subset \R^{d+1}$
such that for $(\delta_d,\ldots, \delta_0)\in R$, there is a real number $r>0$
such that for any $i\not = j$, we have
$$
\varphi_{H_i,\delta_i}(\ol L_i)\cap 
\varphi_{H_j,\delta_j}(\ol L_j) 
\quad
\mbox{lies in the region $|\xi|<r$.}
$$
By a further compactly supported Hamiltonian perturbation,
we may also arrange so that the intersections are transverse.

We consider finite collections of Lagrangian branes $L_0,\ldots, L_d\subset T^*X$ 
to come equipped with such perturbation data,
with the brane structures $(\CE_i, \tilde\alpha_i, \flat_i)$ and taming perturbations $\psi_i$ 
transported via the perturbations. Note that the latter makes sense since the normalized
geodesic flow $\gamma_t$ is an isometry of the metric $g_{con}$.
Then for branes $L_i,L_j$ with $i<j$,
the graded vector space of morphisms between them is defined to be
$$
{hom}_{F(T^*X)}(L_i,L_j) = \bigoplus_{p\in \psi_i(\varphi_{H_i,\delta_i}(L_i))\cap 
\psi_j(\varphi_{H_j,\delta_j}(L_j)) }
{\mathcal Hom}(\mathcal E_i \vert_{p},\mathcal E_j\vert_{p})[-\deg(p)].
$$
where the integer
$\deg(p)$ denotes the Maslov grading of the linear
Lagrangian subspaces at the
intersection.

It is worth emphasizing that near infinity the salient aspect of the above perturbation procedure 
is the relative position of the perturbed branes rather than their absolute position.
The following informal viewpoint
can be a useful mnemonic to keep the conventions straight.
In general, we always think of morphisms as ``propagating forward in time".
Thus 
to calculate the morphisms $\hom_{F(T^*X)}(L_0, L_1)$,
we have required that $L_0,L_1$ are perturbed near infinity by normalized geodesic flow
so that $L_1$ is further in the future than $L_0$. 
But what is important is not that they are both perturbed forward in time,
only that $L_1$ is further along the timeline than $L_0$. So for example,
we could perturb $L_0,L_1$ near infinity by normalized anti-geodesic flow
as long as $L_0$ is further in the past than $L_1$.


\subsubsection{Compositions}

Signed counts of
pseudoholomorphic polygons provide the differential and higher composition maps
of the $A_\infty$-structure. 
We use the following approach of Sikorav~\cite{Sikorav}
(or equivalently, Audin-Lalonde-Polterovich~\cite{ALP})
to ensure that the relevant moduli spaces are compact,
and hence the corresponding counts are finite.

First, as explained in \cite{NZ}, the cotangent bundle $T^*X$ equipped
with the canonical symplectic form $\omega$,
conical almost complex structure $J_{con}$, and conical metric $g_{con}$ is tame
in the sense of ~\cite{Sikorav}. To see this, one can verify that $g_{con}$
is conical near infinity, and so it is easy to derive an upper bound
on its curvature and a positive lower bound on its injectivity radius.

Next, given a finite
collection of branes $L_0,\ldots, L_d$, denote by $L$ the union of their perturbations 
$\psi_i( \varphi_{H_i,\delta_i}(L_i))$ as described above.
By construction, the intersection
of $L$ with the region $|\xi|>r$ is a tame submanifold 
(in the sense of ~\cite{Sikorav})
with respect to the structures
$\omega$, $J_{con}$, and $g_{con}$.
Namely, 
there exists $\rho_L>0$ such
that for every $x\in L$, the set of
points $y\in L$ of distance $d_{}(x,y) \leq \rho_L$ is
contractible, and
there exists $C_L$ giving a two-point distance condition
$d_L(x,y) \leq C_L d_{}(x,y)$  whenever $x,y\in L$ with $d_{}(x,y)<\rho_L$.

Now, consider a fixed topological type of
pseudoholomorphic map
$$
u:(D,\del D) \to (T^*X, L).
$$
Assume that all $u(D)$ intersect a fixed compact region,
and there is an a priori area bound ${\rm Area}(u(D))< A$.
Then as proven in \cite{Sikorav}, one has compactness of the moduli space
of such maps $u$. In fact, one has a diameter bound (depending only on the given constants)
constraining how far the image $u(D)$ can stretch from 
the compact set.

In the situation at hand, for a given $A_\infty$-structure constant, we must consider
pseudoholomorphic maps $u$ from polygons with labeled boundary edges.
In particular, all such maps $u$ have image
intersecting the compact set given by a single intersection point. 
The area of the image $u(D)$ can be expressed
as the contour integral 
$$
{\rm Area}(u(D)) = \int_{u(\del D)} \theta.
$$
Since each of the individual Lagrangian branes making up $L$ is exact,
the contour integral only depends upon the integral of $\theta$ along minimal paths between 
intersection points. Thus such maps $u$ satisfy an a priori
area bound. We conclude that
 for each $A_\infty$-structure constant, 
the moduli space defining the structure constant
is compact, and its points are represented by maps $u$ 
with image bounded by a fixed distance from any of the intersection points.

Finally, as usual, the composition map 
$$
m^d:
{hom}_{F(T^*X)}(L_0,L_1)\otimes\dots\otimes{hom}_{F(T^*X)}(L_{d-1},L_d)\rightarrow
{hom}_{F(T^*X)}(L_0,L_d)[2-d]
$$ 
is defined as follows.  Consider elements 
$p_i\in {hom}(L_i,L_{i+1}),$
for $i=0,\ldots,d-1$,
and $p_d\in hom(L_0,L_d)$.
Then the coefficient of $p_d$ in 
$m^d(p_0,\dots,p_{d-1})$
is defined to be the signed sum over pseudoholomorphic maps from a disk
with $d+1$ counterclockwise cyclically ordered
marked points mapping to the $p_i$ and corresponding
boundary arcs mapping to the perturbations of $L_{i+1}.$  Each map contributes
according to the holonomy of its boundary, where adjacent perturbed
components $L_i$ and $L_{i+1}$ are glued with $p_i.$

Continuation maps with respect to families of perturbed branes ensure the consistency of
all of our definitions. For a recollection of the basic idea, see Section~\ref{sect inv} below.

One should also note that it is not immediately evident that we have a well-defined $A_\oo$-category. Bounding the behavior of each moduli space involves fixing the complexity of the input data of which branes, intersection points, and structure constants are in play. There are potentially (at least) two more or less equivalent ways to proceed. The first is very formal and somewhat standard: as explained by Stasheff~\cite{S}, there is a family of operads called $A_n$ which parameterize ``partial" $A_\oo$-structures. Compatible $A_n$-categories, for all $n$, provide an $A_\oo$-category since the $A_\oo$-operad is the union of the $A_n$-operads. 
Second, and more geometrically,  to make any given calculation, we may need to insist upon smaller and smaller perturbation data. This leads one to think of each object as the ``limit" of a brane under smaller and smaller perturbation data.
Either formulation can be implemented with the invariance of Floer calculations established in~\cite{NZ, N}, and reviewed in Section~\ref{sect inv} below. For an alternative geometric (rather than homotopical) approach, one could consult Oh's paper~\cite{Oh} for a detailed analysis of the relevant pseudoholomorphic disk theory.

\medskip

Consider the dg category of right modules over the Fukaya category of $T^*X$.
Throughout this paper, 
we write $F(T^*X)$ for the
the full subcategory
of twisted complexes of representable modules,
and refer to it  as the triangulated envelope of the Fukaya category.
We use the term Lagrangian brane to refer to an object of the Fukaya category,
and brane to refer to an object of its triangulated envelope $F(T^*X)$.


\subsection{Microlocalization}\label{sect micro}

We review here the microlocalization quasi-equivalence
constructed in~\cite{NZ}. 
Some useful notation: for a function $m:X\to \R$ and number $r\in R$, we write $X_{m = r}$
for the subset $\{x\in X | m(x) = r\}$ and similarly for inequalities.

\medskip

Let $i: U\hra X$ be an open
submanifold that is a $\CC$-subset of $X$.
Since the complement $X\setminus U$ is a closed $\CC$-subset of $X$, 
we can find a non-negative function
$m:X\to \R_{\geq 0}$ such that $X\setminus U$ is precisely the zero-set of $m$.
Since the complement of the critical values of $m$ form an open $\CC$-subset of $\R$,
 the subset $X_{m >\eta}$
is an open submanifold with smooth hypersurface boundary $X_{m =\eta} $,
for any sufficiently small $\eta >0$.

\medskip

Now let $i_{\alpha}: U_\alpha\hra X$, for $\alpha=0,\ldots, d$,
be a finite collection of open
submanifolds that are $\CC$-subsets of $X$.
Fix non-negative function
$m_\alpha:X\to \R_{\geq 0}$,  for $\alpha=0,\ldots, d$, 
such that $X\setminus U_\alpha$ is precisely the zero-set of $m_\alpha$.
There is a fringed set $R\subset \R^{d+1}_+$ such that 
for any $(\eta_d,\ldots, \eta_0)\in R$, the following holds. First the hypersurfaces 
$X_{m_\alpha =\eta_\alpha}$ are all transverse. Second, for $\alpha<\beta$,
there is a quasi-isomorphism 
of complexes
$$
\hom_{Sh_c(X)}(i_{\alpha*}\C_{U_\alpha}, i_{\beta*}\C_{U_\beta})
\simeq
(\Omega(X_{m_\alpha\geq \eta_\alpha}\cap X_{m_\beta>\eta_\beta}, 
X_{m_\alpha=\eta_\alpha}\cap X_{m_\beta>\eta_\beta}),d)
$$
where $(\Omega, d)$ denotes the relative de Rham complex which calculates the
cohomology of the pair. Furthermore, the composition of morphisms in $Sh_c(X)$
corresponds to the wedge product of forms.

\medskip

Next let 
$f_\alpha:X_{m_\alpha > \eta_\alpha}\to \R$,  for $\alpha=0,\ldots, d$, 
be the logarithm $f_\alpha = \log m_\alpha$.
While choosing the sequence of parameters $(\eta_d, \ldots, \eta_0)$, we can also
choose a sequence of small positive parameters $(\epsilon_d, \ldots, \epsilon_0)$ such that 
the following holds. 
For any $\alpha <\beta$, consider the open submanifold 
$X_{m_\alpha >\eta_\alpha,m_\beta>\eta_\beta} = X_{m_\alpha >\eta_\alpha}\cap 
X_{m_\beta>\eta_\beta} $ with corners equipped with the function
$f_{\alpha, \beta} = \epsilon_\beta f_\beta - \epsilon_\alpha f_\alpha$.
Then there is an open set of Riemannian metrics on $X$ such that for all $\alpha<\beta$,
it makes sense to consider the Morse complex
$\CM(X_{m_\alpha >\eta_\alpha,m_\beta>\eta_\beta}, f_{\alpha,\beta})$,
and
there is a quasi-isomorphism
$$
(\Omega(X_{m_\alpha\geq \eta_\alpha}\cap X_{m_\beta>\eta_\beta}, 
X_{m_\alpha=\eta_\alpha}\cap X_{m_\beta>\eta_\beta}),d)
\simeq
\CM(X_{m_\alpha >\eta_\alpha,m_\beta>\eta_\beta}, f_{\alpha,\beta}).
$$
Furthermore, homological perturbation theory provides a quasi-isomorphism
between the $A_\infty$-composition structure on the collection of Morse complexes
and the dg structure given by the wedge product of forms.

\medskip

Finally, we 
define the microlocalization functor
$$
\xymatrix{
\mu_X:\Sh_c(X)\ar[r] & F(T^*X)
}
$$ 
as follows. The standard objects
$i_*\C_U$ associated to open submanifolds $i:U\hra X$ generate the constructible
dg derived category $Sh_c(X)$. Thus to construct $\mu_X$, it suffices to find a 
parallel collection of standard objects of $F(T^*X)$.

Given an open submanifold $i:U\hra  X$
and function $m:X\to\R_{\geq 0}$ with zero-set the complement $X\setminus U$,
define
the {standard Lagrangian} $L_{U,f*}\hra T^*X\vert_U $ to be
the graph
$$
L_{U,f*}= \Gamma_{df},
$$
where $df $ denotes the differential of the logarithm $f=\log m$.

The standard Lagrangian $L_{U,f*}$ comes
equipped with a canonical brane structure $(\CE, \tilde\alpha, \flat)$ and taming perturbation $\psi$. 
Its flat vector bundle $\CE$ is trivial, and
its grading $\tilde \alpha$ and relative pin structure $\flat$ are the canonical structures on a graph.
Its taming perturbation $\psi$ is given by the family of standard Lagrangians 
$$
L_{X_{m = \eta}, f_\eta *} = \Gamma_{df_\eta},
\quad
\mbox{ for sufficiently small $\eta>0$},
$$
where $f_\eta =\log m_\eta$ is the logarithm of the shifted function $m_\eta = m -\eta$.

Now one can extend the fundamental result of Fukaya-Oh~\cite{FO} identifying Morse
moduli spaces and Fukaya moduli spaces to the current setting. Namely, one can show
that for any finite ordered collection of open
submanifolds
$i_{\alpha}: U_\alpha\hra X$, for $\alpha=0,\ldots, d$, and any finite collection
of $A_\infty$-compositions respecting the order, there is a fringed set $R\subset \R^{d+1}$
such that for any parameters $(\eta_d,\ldots, \eta_0)\in R$,
the Morse moduli spaces
of the ordered collection of functions $f_{\eta_\alpha}$ are isomorphic
to the Fukaya moduli spaces of the ordered collection of standard branes $L_{X_{m=\eta}, f_\eta*}$
(after further variable dilations of the functions and branes).

Once and for all, for each $U\subset X$, let us choose 
a non-negative  
function $m:X\to \R_{\geq 0 }$ such that the complement $X\setminus U$
is the zero-set of $m$. 
We denote the resulting standard brane $L_{U, f*}$, where $f = \log m$, by the abbreviated notation $L_{U*}$.
We define the functor $\mu_X$ so that on objects we have
$$
\mu_X(i_*\C_U) = L_{U*}
$$ 

The arguments of \cite{NZ} outlined above show that calculations among standard branes are equivalent to calculations among corresponding standard sheaves. In particular, given $U\subset X$, there is no preference as to which function
$m:X\to \R_{\geq 0}$ is used, and hence no preference as to which 
 standard brane 
 is used. It follows that any standard brane for $U\subset X$  will have the same structure with respect to other standard branes. 
For example, a standard brane and its dilations will have the same structure with respect to other standard branes.

\begin{thm}[\cite{N, NZ}]\label{thm equiv}
Microlocalization is a quasi-equivalence
$$
\xymatrix{
\mu_X: Sh_c(X) \ar[r]^-{\sim} & F(T^*X).
}
$$
\end{thm}

 One can rephrase the equivalence of the theorem to say that every brane can be expressed in terms of standard branes. 
 Thus to understand properties of branes it suffices to study collections of standard branes.
It follows from the discussion immediately preceding the theorem that all of  the standard branes for a given subset are equivalent to each other (since the arguments of ~\cite{NZ} show that they lead to the same calculations as the unambigiuous standard sheaves), and
all branes are equivalent to themselves under dilation (since the dilation of a standard brane is equally well a standard brane and thus leads to the same calculations).

The theorem admits the following refinement.
Given a  conical Lagrangian $\Lambda\subset T^*X$, let $F_\Lambda(T^*X) \subset F(T^*X)$
denote the full subcategory generated by branes $L\subset T^*X$ whose boundary
 $L^\infty=\ol L \cap T^\infty X$
lies in the boundary  $\Lambda^\infty=\ol \Lambda\cap T^\infty X$.

For any stratification $\CS=\{S_\alpha\}$ 
of $X$, let $\Lambda_\CS\subset T^*X$ denote the union of conormal bundles 
$\Lambda_\CS = \cup_\alpha T^*_{S_\alpha} X$.
By construction, the microlocalization $\mu_X$ takes the full subcategory $Sh_\CS(X)\subset Sh_c(X)$
to the full subcategory $F(T^*X)_{\Lambda_\CS}\subset F(T^*X)$. 

Conversely, given an object $L$ of $F(T^*X)$,
let $\Lambda\subset T^*X$ be a conical Lagrangian such that the boundary
  $L^\infty$
 lies in the boundary $\Lambda^\infty$ (for instance, one can minimally take $\Lambda$ to be
the cone over $L^\infty$).
Then for any object $\CF$ of $Sh_c(X)$
such that $\mu_X(\CF)\simeq L$, and
for any stratification $\CS=\{S_\alpha\}$ 
of $X$ such that $\Lambda \subset \Lambda_\CS = \cup_\alpha T^*_{S_\alpha} X$,
the object $\CF$ belongs to $Sh_{\CS}(X)$.

\medskip

One proves the theorem and the above refinement by studying non-characteristic
families of branes.
By a
one-parameter family of closed (but not necessarily compact) 
submanifolds (without boundary) in $T^*X$, we mean a closed 
 submanifold
$$
\fL \hra \R\times T^*X
$$ 
satisfying the following:

\begin{enumerate}

\item 
The restriction of the projection $p_\R: \R\times T^*X\to \R$
to the submanfold $\fL$ is nonsingular.

\item There is a real number $r>0$, such that
the restriction of the product  
$
p_\R\times|\xi| : \R\times T^*X \to \R \times [0,\infty)
$
to the subset $\{|\xi|>r\}\cap \fL$ is proper and nonsingular. 

 \item There is a compact interval $[a,b]\hra \R$ such that
 the restriction of the projection $p_X:\R\times T^*X\to T^*X$ 
 to the submanifold $p_\R^{-1}([\R\setminus [a,b])\cap \fL$
 is locally constant.

\end{enumerate}

Note that conditions (1) and (2) will be satisfied if 
the restriction of the projection $\ol p_\R:\R\times \ol T^*X \to \R$ 
to the closure $\ol\fL\hra \ol T^*X$ is 
nonsingular as a stratified map, but the weaker condition stated is a useful generalization.
It implies in particular that the fibers $\fL_s = p_\R^{-1}(s) \cap\fL \hra T^*X$ are all diffeomorphic, but imposes no requirement
that their boundaries at infinity should all be homeomorphic as well.

\medskip

By a one-parameter family of tame Lagrangian branes in $T^*X$, we mean
a one-parameter family of closed submanifolds $\fL\hra \R\times T^*X$ in the above sense
such that the fibers $\fL_s = p_\R^{-1}(s) \cap\fL \hra T^*X$ also satisfy:

\begin{enumerate}

\item The fibers $\fL_s$ are exact tame Lagrangians with respect to the usual
symplectic structure and any almost complex structure conical near infinity.

\item The fibers $\fL_s$ are
equipped with a locally constant brane structure $(\CE_s, \tilde\alpha_s,\flat_s)$
with respect to the usual background classes. 

\end{enumerate}

Note that if we assume that $\fL_0$ is an exact Lagrangian, then $\fL_s$ being an
exact Lagrangian
is equivalent to the family $\fL$ being given by the flow $\varphi_{H_s}$
of the vector field of a time-dependent
Hamiltonian $H_s:T^*X\to \R$.
Note as well that a brane structure consists of topological data, so can be transported unambiguously
along the fibers of such a family. 

\medskip

Fix a conical Lagrangian $\Lambda\subset T^*X$,  
with boundary $\Lambda^\infty=\ol \Lambda\cap T^\infty X$.
As above,  let $F_\Lambda(T^*X)$ be the full
subcategory of $F(T^*X)$ generated by Lagrangian branes $L$ 
whose boundary $L^\infty=\ol L\cap T^\infty X$ lies in $\Lambda^\infty$. 

Suppose $\fL\hra \R\times T^*X$ is a one-parameter family of tame Lagrangian branes.
We will say that $\fL$ 
is $\Lambda$-non-characteristic if 
$$
\ol \fL_s\cap \Lambda^\infty =\emptyset,
\qquad
\mbox{ for all $s\in\R$.}
$$

\begin{prop}[\cite{N}]\label{prop brane non-char}
Suppose $\fL\hra \R\times T^*X$ is a $\Lambda$-non-characteristic
one-parameter family of tame Lagrangian branes.
For any test object $P$ of $F_\Lambda(T^*X)$, there are functorial quasi-isomorphisms among the 
Floer complexes
$$
\hom_{F(T^*X)}(P, \fL_{s}),
\qquad
\mbox{ for all $s\in\R$}.
$$
\end{prop}

The proof of the proposition is very general and does not use that $X$ is compact
in any serious way. For example, it holds when $X$ is complete, or in fact for tame Lagrangian branes
in more general exact symplectic targets. We will use it in later sections for the contangent bundle
of a vector space.


\subsection{Floer invariance}\label{sect inv}
In the definition of the Fukaya category of $T^* X$ recalled in Section~\ref{sect Fukaya category}, as well as in the 
microlocalization quasi-equivalence recalled in Section~\ref{sect micro}, we have appealed to results of ~\cite{NZ, N} on the invariance of Floer calculations under suitable motions of noncompact branes.
To make the current paper as self-contained as possible, we include here a brief section reviewing the (somewhat ad hoc) arguments which establish the following basic example of this invariance. One could also consult Oh's paper~\cite{Oh} which contains a detailed analysis of the canonical structures provided by pseudoholomorphic disk theory.

\begin{prop}\label{prop inv}
Suppose $\fL_s$ is a family of objects of $F(T^*X)$. Suppose $L'$ is a fixed test object which is disjoint from $\fL_s$ near infinity for all $s$. Suppose $\fL_s$ is transverse to $L'$ except for finitely many points.

Then for any $a,b$ with $\fL_a$ and $\fL_b$ transverse to $L'$, the Floer chain complexes
$CF(\fL_{a}, L')$ and $CF(\fL_{b}, L')$ are quasi-isomorphic.

\end{prop}

Before proving the proposition in full, it is convenient to first prove the following special case.

\begin{lem}
Suppose $\fL_s$ is a family of objects of $F(T^*X)$. Suppose $L'$ is a fixed test object which is disjoint from $\fL_s$ near infinity for all $s$. 

Fix $s_0$ and assume $\fL_{s_0}$ is transverse to $L'$. Then there is an $\epsilon >0$ so that for all $s_1 \in (s_0 - \epsilon, s_0+\epsilon)$, the Floer chain complexes
$CF(\fL_{s_0}, L')$ and $CF(\fL_{s_1}, L')$ are quasi-isomorphic.

\end{lem}

\begin{proof}
By our assumptions on the tame behavior (in the sense of ~\cite{Sikorav}) of $\fL_{s_0}$ and $L'$ near infinity, the moduli spaces giving the differential of $CF(\fL_{s_0}, L')$ are compact. This follows from the {a priori} $C^0$-bound: there is some $r_0\gg 0$, such that no disk in the moduli space leaves the region $|\xi| < r_0$, where $(x, \xi)$ are local coordinates on $T^*X$, and $|\xi|$ is the Riemannian metric.

Choose some $r_1 > r_0$. Then for very small $\epsilon> 0$ and any $s_1 \in (s_0 - \epsilon, s_0+\epsilon)$, we may decompose the motion $\fL_{s_0} \rightsquigarrow \fL_{s_1}$ into two parts: first, a motion $\fL_{s_0} \rightsquigarrow L$ {\em supported in the region $|\xi |> r_0$}; and then second, a {\em compactly supported} motion $L \rightsquigarrow \fL_{s_1}$.  We must show that each of the above two motions leads to a quasi-isomorphism.

First, for the motion $\fL_{s_0} \rightsquigarrow L$, since we have not changed $\fL_{s_0}$ or $L'$ in the region $|\xi| < r_0$, the same {a priori} $C^0$-bounds  of \cite{Sikorav} hold (they only depend on the Lagrangians in the region $|\xi| < r_0$), and the pseudoholomorphic strips for the pair $(\fL_{s_0}, L')$ and for the pair $(L, L')$ are in fact {exactly the same} (we could perversely attach ``wild" non-intersecting ends to either and it would not make a difference.) Thus we can take the ``continuation map" to be the identity.

(One should probably not use the term ``continuation map" for such a construction. 
Rather, it is an example of the more general setup of parameterized
moduli spaces.
In the above setting, one can obtain a uniform $C^0$-bound over the family, so the parameterized
moduli space is compact, and hence one can apply standard cobordism
arguments to prove the matrix coefficients at the initial and final
time are the same. We thank an anonymous referee for this perspective on the argument.)

Second, the motion $L \rightsquigarrow \fL_{s_1}$ is compactly supported, so standard PDE techniques provide a continuation map.
\end{proof}

\begin{proof}[Proof of Proposition~\ref{prop inv}]
By the previous lemma, it suffices to show that for any $s_0$ with $\fL_{s_0}$ not (necessarily) transverse to $L'$, there is a small $\epsilon>0$ such that the Floer chain complexes
$CF(\fL_{s_0 -\epsilon}, L')$ and $CF(\fL_{s_0+\epsilon}, L')$ are quasi-isomorphic.

To see this, let $H_s(x, \xi)$ be a (time-dependent) Hamiltonian giving the motion $\fL_s$. Choose a bump function $b(|\xi|)$ which is $0$ near infinity and $1$ on a compact set containing all of the (possibly non-transverse) intersection points $\fL_{s_0} \cap L'$. 

The product Hamiltonian $\tilde H(x, \xi) = b(|\xi|) H_s(x, \xi)$ gives a family $\tilde \fL_s$ through the base object $\fL_{s_0}$ satisfying: (1) $\tilde \fL_s$ is transverse to $L'$ whenever $|s-s_0| $ is small and nonzero, and (2) $\tilde \fL_s$ is equal to $\fL_{s_0}$ near infinity.
Therefore since the motion of $\tilde \fL_s$ is compactly supported, standard PDE techniques provide a continuation map giving a quasi-isomorphism between 
$CF(\tilde \fL_{s_0 -\epsilon}, L')$ and $CF(\tilde \fL_{s_0 + \epsilon}, L')$, for small enough $\epsilon >0$.

Finally, returning to the bump function $b(|\xi|)$, one can construct motions $\fL_{s_0 - \epsilon} \rightsquigarrow \tilde \fL_{s_0 - \epsilon}$ and 
 $\tilde \fL_{s_0 +\epsilon} \rightsquigarrow  \fL_{s_0 + \epsilon}$ which are supported near infinity and thus in particular always  transverse to $L'$. Thus we may apply the previous lemma to  obtain quasi-isomorphisms between 
$CF(\fL_{s_0 -\epsilon}, L')$ and $CF( \tilde \fL_{s_0 - \epsilon}, L')$, 
and similarly, between $CF(\tilde \fL_{s_0 +\epsilon}, L')$ and $CF( \fL_{s_0 + \epsilon}, L')$.
Putting together the above, we obtain a quasi-isomorphism between 
$CF(\fL_{s_0 -\epsilon}, L')$ and $CF( \fL_{s_0 + \epsilon}, L')$.
\end{proof}

\begin{rmk}
The above proposition (which is a condensed form of arguments of ~\cite{N, NZ}) is closely related to Question 1.3 of Oh's paper~\cite{Oh} which asks whether a homology-level continuation map constructed by a careful limiting argument with PDE techniques is induced by a chain-level morphism. While we have not investigated this, it is not hard to believe that the quasi-isomorphism of the above proposition provides the desired lift.
\end{rmk}



\section{Fourier transform for branes}\label{sec fourier for branes}

In this section, we study the symplectic topology
of the cotangent bundle of a real finite-dimensional vector space $V$.
Our aim is to describe a Fukaya theory of branes in $T^* V \simeq V\times V^*$ 
that treats the horizontal and vertical directions as symmetrically as possible.
Wherever possible, we will appeal to arguments of the preceding section
and restrict the discussion here to the new aspects which arise.


\subsection{Preliminaries}\label{sect main prels}
Fix a real finite dimensional vector space $V$.

We will write $V_1$ in place of $V$,
and $V_2$ for its dual $V^*$. 
Let $\R^n_x$ denote standard Euclidean space
with coordinate $x=(x_i)$, and let $\R^n_\xi$ denote the dual Euclidean space with coordinate 
$\xi=(\xi_i)$,
 so that $\langle x, \xi\rangle = \sum_{i=1}^n x_i \xi_i$.
By choosing an isomorphism $V_1\simeq \R^n_x$,
we obtain a dual isomorphism $V_2\simeq \R^n_\xi$. 
For concreteness,
we will often assume such identifications have been fixed
(though our constructions will not depend on the specific identifications).

Let $\omega_1$, $\omega_2$ denote the respective canonical exact symplectic forms
on $T^*V_1$, $T^*V_2$. 
Under the canonical identifications
$$
T^*V_1\simeq V_1\times V_2 \simeq T^*V_2,
$$ 
the canonical exact symplectic forms are related by
$\omega_1= -\omega_2,
$
since in local coordinates, we have
$$
\omega_1=\sum_{i=1}^n d\xi_idx_i 
\qquad
\omega_2 = \sum_{i=1}^n dx_i d\xi_i.
$$
In what follows, unless otherwise stated, we will break symemtry and
work with 
the symplectic structure $\omega_1$. Thus to identify $V_1\times V_2$ and $T^*V_2$
as symplectic manifolds, we will compose the above canonical identification
with the negation map on the first factor: $x\mapsto -x$, $\xi\mapsto \xi$.
When it is not clear from context, we will write 
$$
\iota:T^*V_2\risom V_1\times V_2
$$
for the symplectic identification.

Given a positive definite quadratic form on $V_1$, we obtain an identification $V_1\simeq V_2$.
For vectors $v_1\in V_1$, $v_2\in V_2$,  we write $|v_1|$, $|v_2|$ for the 
respective lengths of $v_1$, $v_2$. 


\subsubsection{Symmetric compactification}

To control noncompact Lagrangians in $V_1\times V_2$,
we will work with a symmetric product compactification.

Given a vector space $V$, consider the spherical compactification
$$
\ol V = (V \times \R_{\geq 0} \setminus \{(0,0)\}) / \R^+
$$
where
$\R_{+}$ acts 
by dilations on both factors.
The canonical inclusion $V\hookrightarrow \ol V$ sends
a vector $v$ to the class of $[v,1]$.
The boundary sphere at infinity 
$V^\infty =\ol V\setminus V
$ 
consists of classes of the form $[v, 0]$ with $v$ a non-zero vector.

\medskip

Now let $\ol V_1$, $\ol V_1$ be the spherical compactifications of $V_1$, $V_2$ with spheres at infinity 
$V_1^\infty$, $V^\infty_2$.
We will work with the symmetric product compactification $\ol V_1\times \ol V_2$.
Its boundary at infinity is the disjoint union of a codimension one boundary
$$
B= (V_1 \times V_2^\infty) \coprod (V_1^\infty\times V_2),
$$
along with a codimension two corner
$$
C=V_1^\infty \times V_2^\infty.
$$


\subsubsection{Symmetric almost complex structure}
To control holomorphic disks in $V_1\times V_2$, 
we will work with a symmetric almost complex structure.

Fix a positive definite quadratic form on $V_1$, and let $j_0:V_1\risom V_2$ be the corresponding
identification. 

Fix $0<r_1< r_2$, and a smooth increasing function $b:\R\to \R$ satisfying $b(r) = 0$, for $r<r_1$,
and $b(r) = 1$,  for $r_2<r$.
Consider the functions
$$
w(v_1, v_2)=\frac{1 + |v_2|}{1+|v_1|}
\qquad
\rho(v_1, v_2) = |v_1|^2 + |v_2|^2
$$
and define the $\omega_1$-compatible almost complex structure
$$
J_{sym}=
\left(
\begin{matrix}
0 & w^{-b(\rho)} j_0^{-1} \\
-w^{b(\rho)} j_0 & 0
\end{matrix}
\right)
\in\on{End}(T(V_1 \oplus V_2)).
$$
We refer to $J_{sym}$ as a symmetric (asymptotically) conical almost complex structure.
The corresponding metric $g_{sym}(v,v) = \omega_1(v, J_{sym} v)$ is complete and tame.



\subsection{From branes to sheaves}
In this section, we explain how to associate 
constructible sheaves to branes in $V_1\times V_2$.

\subsubsection{Branes in $V_1\times V_2$}
To define Lagrangian branes in $V_1\times V_2$,
we must first fix background structures.
Of course, the background class for relative pin structures
is trivial since it lies in $H^2(V_1\times V_2, \Z/2\Z)\simeq 0$.
Thus the only background structures of note are the bicanonical bundle and its trivialization. 
We will break symmetry and work with the bicanonical bundle $\kappa_1^{\otimes 2}$
and trivialization $\eta_1^2$ coming from
the canonical identification $T^*V_1\simeq V_1\times V_2$.

\medskip

By a Lagrangian brane $L\hra V_1\times V_2$, we mean a four-tuple
$
(L, \CE, \tilde\alpha,\flat)
$
consisting of an exact (not necessarily compact) closed Lagrangian submanifold 
${L} \hra V_1\times V_2$
equipped with a brane structure:
this includes a flat vector bundle $\CE\to L$,
along with a grading $\tilde\alpha:L\to \R$ 
(with respect to the bicanonical trivialization $\eta_1^2$ of the bicanonical
bundle $\kappa_1^{\otimes 2}$) and a pin structure $\flat$.

Furthermore, we place two assumptions on the
Lagrangian $L$.
Recall the symmetric product compactification $\ol V_1\times \ol V_2$,
and the symmetric conical almost complex structure $J_{sym}$
and corresponding metric $g_{sym}$.
First, 
we assume that
the closure $\ol L \hra  \ol V_1\times \ol V_2$ is a $\CC$-subset.
Second, 
we assume the existence of a perturbation $\psi$ that moves the initial
Lagrangian $L$ to a nearby Lagrangian tame (in the sense of~\cite{Sikorav})
with respect to the symmetric conical metric $g_{sym}$.

\medskip

Let us take a moment to comment on the asymmetry of the above definition.
Recall that to identify $T^*V_2$  and $V_1\times V_2$ as symplectic manifolds, 
we compose the canonical
identification
with the negation map on the first factor: $x\mapsto -x$, $\xi\mapsto \xi$.
Thus the bicanonical bundle $\kappa_2^{\otimes 2}$
coming from  the resulting identification $\iota:T^*V_2\risom V_1\times V_2$
is canonically identified with $\kappa_1^{\otimes 2}$.
The bicanonical trivialization $\eta_2^2$ 
arising via $\iota$
satisfies $\eta_2^2=-\eta_1^2$.
We will identify the two  bicanonical trivializations via  the path 
$\eta_1^2\rightsquigarrow \eta_2^2$ induced by the positively-oriented path
 $1 \rightsquigarrow-1$ inside of $\C^\times$.

\begin{ex}
Suppose $V_1=  \R_x$, $V_2=\R_\xi$, and let $L\hra \R_x\times \R_\xi$ be the Lagrangian 
$L = \{x\xi= 1\}$. Then the canonical grading of $L$ as a graph in $T^*\R_x$
coincides with its canonical grading as a graph in $T^*\R_\xi$. Note that if we write
$\eta$ for the dual coordinate of $\xi$, then as a graph in $T^*\R_\xi \simeq \R_\xi\times\R_\eta$,
we have $L=\{\xi\eta = -1\}$.
\end{ex}


\subsubsection{Fukaya $A_\infty$-structures}

Suppose we have a collection of Lagrangian branes in $V_1\times V_2$
that are pairwise transverse and whose boundaries at infinity are disjoint.

We define a  Fukaya pre-$A_\infty$-category structure (or partially define $A_\infty$-category structure)
on the collection as follows.
The graded vector space of morphisms between distinct branes is defined to be
$$
{hom}_{}(L_0,L_1) = \bigoplus_{p\in L_0\cap L_1}
{\mathcal Hom}(\mathcal E_0 \vert_{p},\mathcal E_1\vert_{p})[-\deg(p)].
$$
where the integer
$\deg(p)$ denotes the Maslov grading of the linear
Lagrangian subspaces at the
intersection. 

 Signed counts of
pseudoholomorphic polygons provide the differential and higher composition maps
of the pre-$A_\infty$-structure. 
To ensure that the relevant moduli spaces are compact,
we appeal to the same arguments used for cotangent bundles
in~Section~\ref{sect Fukaya category}.
First, the target $V_1\times V_2$ 
with the symplectic form $\omega_1$,
almost complex structure $J_{sym}$, and corresponding metric $g_{sym}$ is tame.
Next, let $L\hra V_1\times V_2$ be the union of any finite number of branes
from the collection.
By assumption, the intersection
of $L$ with the region $\rho(x, \xi) = |x|^2 + |\xi|^2>r$, for large $r>0$, is a tame submanifold.
Since our branes are exact,
for a given $A_\infty$-structure constant, the relevant
pseudoholomorphic maps
$$
u:(D,\del D) \to (V_1\times V_2, L),
$$
satisfy an a priori
area bound.
Thus we have a diameter bound 
on their images $u(D)$, and hence the moduli space of all such maps is compact.

Finally, for distinct branes, the composition maps
$$
m^d:
{hom}_{F(T^*X)}(L_0,L_1)\otimes\dots\otimes{hom}_{F(T^*X)}(L_{d-1},L_d)\rightarrow
{hom}_{F(T^*X)}(L_0,L_d)[2-d]
$$ 
is defined as follows.  Consider elements 
$p_i\in {hom}(L_i,L_{i+1}),$
for $i=0,\ldots,d-1$,
and $p_d\in hom(L_0,L_d)$.
Then the coefficient of $p_d$ in 
$m^d(p_0,\dots,p_{d-1})$
is defined to be the signed sum over pseudoholomorphic maps from a disk
with $d+1$ counterclockwise cyclically ordered
marked points mapping to the $p_i$ and corresponding
boundary arcs mapping to the perturbations of $L_{i+1}.$  Each map contributes
according to the holonomy of its boundary, where adjacent perturbed
components $L_i$ and $L_{i+1}$ are glued with $p_i.$

\medskip

Our key technical tool for understanding calculations in $F(V_1\times V_2)_{pre}$
is their invariance under certain motions of branes. The following is a direct generalization
of Proposition~\ref{prop brane non-char} and the discussion preceding it.

By a one-parameter family of closed (but not necessarily compact) 
submanifolds (without boundary) in $V_1\times V_2$, we mean a closed 
 submanifold
$$
\fL \hra \R\times V_1\times V_2
$$ 
satisfying the following:

\begin{enumerate}

\item 
The restriction of the projection $p_\R: \R\times V_1\times V_2 \to \R$
to the submanfold $\fL$ is nonsingular.

\item There is a real number $r>0$, such that
the restriction of the product
$
p_\R\times \rho: V_1\times V_2 \to \R \times [0,\infty)
$
to the subset $\{\rho>r\}\cap \fL$ is proper and nonsingular.

 \item There is a compact interval $[a,b]\hra \R$ such that
 the restriction of the projection $p_X:\R\times T^*X\to T^*X$ 
 to the submanifold $p_\R^{-1}([\R\setminus [a,b])\cap \fL$
 is locally constant.

\end{enumerate}

By a one-parameter family of tame Lagrangian branes in $V_1\times V_2$, we mean
a one-parameter family of closed submanifolds $\fL\hra \R\times V_1\times V_2$ in the above sense
such that the fibers $\fL_s = p_\R^{-1}(s) \cap\fL \hra V_1\times V_2$ 
are exact tame Lagrangians equipped with a locally constant brane structure $(\CE_s, \tilde\alpha_s,\flat_s)$.

Fix a biconical Lagrangian $\Lambda\subset V_1\times V_2$,  
with boundary $\Lambda^\infty$.
Let $F_\Lambda(V_1\times V_2)_{pre}$ be the
pre-$A_\infty$-category of Lagrangian branes $L$ 
whose boundary $L^\infty$ lies in $\Lambda^\infty$. 
Suppose $\fL\hra \R\times V_1\times V_2$ is a one-parameter family of tame Lagrangian branes.
We will say that $\fL$ 
is $\Lambda$-non-characteristic if 
$$
\ol \fL_s\cap \Lambda^\infty =\emptyset,
\qquad
\mbox{ for all $s\in\R$.}
$$

As with Proposition~\ref{prop brane non-char},
one can repeat the proof from~\cite{N} 
to establish the following assertion.
In fact, the same argument gives a further generalization for tame Lagrangian branes
in other exact symplectic targets. 

\begin{prop}\label{prop brane non-char in general}
Suppose $\fL\hra \R\times V_1\times V_2$ is a $\Lambda$-non-characteristic
one-parameter family of tame Lagrangian branes.
For any test object $P$ of $F_\Lambda(V_1\times V_2)_{pre}$, there are functorial quasi-isomorphisms among the 
Floer complexes
$$
\hom_{F(V_1\times V_2)_{pre}}(P, \fL_{s}),
\qquad
\mbox{ for all $s\in\R$}.
$$
\end{prop}


\subsubsection{Functionals on sheaves}

Let us single out a special classe of branes in $V_1\times V_2$. 
Consider the situation
from the perspective of $V_1$
so that we have $T^* V_1\simeq V_1\times V_2$. 
We say that a brane $L$ is compact along the first factor if its
 projection to $V_1$ is compact, or equivalently, the closure $\ol L$
lies in $\ol T^* V_1\simeq V_1\times \ol V_2$.

Consider the collection of branes in $V_1\times V_2$
that are compact along the first factor. Then we can regard them as branes in $T^* V_1$,
and accordingly define an honest Fukaya $A_\infty$-category structure on them
repeating  
our constructions for cotangent bundles
in~Section~\ref{sect Fukaya category}.
In particular, we can find a fringed set parametrizing controlled Hamiltonian perturbations
that move the branes so that they do not intersect at infinity.
We write $F(T^*V_1)_\kappa$ for the triangulated envelope of the Fukaya category
of such branes.

Consider the full subcategory  $Sh_c(V_1)_\kappa \subset Sh_c(V_1)$ 
of compactly supported objects.
Since $V_1$ is complete (though noncompact), we can repeat
the construction of microlocalization to obtain a quasi-embedding
$$
\mu_{V_1}: Sh_c(V_1)_\kappa \hra F(V_1)_\kappa.
$$

\medskip

Now fix a brane $L\hra V_1\times V_2$, without any assumption on whether
it is compact in either direction. 
Let us measure the structure of $L$
using the quasi-embedding $\mu_{V_1}$.
Namely, we can consider the right module
$$
\tilde \pi_{V_1}(L) : Sh_c(V_1)^{op}_\kappa \to \Ch
\qquad
\tilde \pi_{V_1}(L)(\CF) = \hom_{F(V_1\times V_2)_{pre}}(\mu_{V_1}(\CF), L)
$$
By definition, if the boundary of $\mu_{V_1}(\CF)$ intersects the boundary of $L$,
then we simply perturb the former according to our usual conventions
for cotangent bundles. Thus we can always unambiguously make the 
necessary calculations to define an honest module.

By Proposition~\ref{prop brane non-char} and the results on
quasi-representability from~\cite{N}, the module $\tilde \pi_{V_1}(L)$ is quasi-repesented
by an object of $Sh_c(V_1)$ which we denote by $\pi_{V_1}(L)$.
Given a relatively compact open submanifold $i:U\hra V_1$, we have quasi-isomorphisms
of complexes
$$
 \pi_{V_1}(L)(U) \simeq \hom_{Sh_c(V_1)}(i_!\C_U, \CF)
\simeq \hom_{F(V_1\times V_2)_{pre}}(L_{U!}\otimes or_{V_1}[-\dim V_1], L).
$$
Fix a conical (with respect to the second factor) Lagrangian $\Lambda\hra T^*V_1$ such that the part of the boundary
$L^\infty$ that lies in $V_1\times V^\infty_2$ in fact lies in the boundary $\Lambda^\infty$.
Then for any stratification $\CS=\{S_\alpha\}$ 
of $V_1$ such that $\Lambda \subset \Lambda_\CS = \cup_\alpha T^*_{S_\alpha} V_1$,
the object $ \pi_{V_1}(L)$ lies in $Sh_{\CS}(V_1)$.

\medskip

Similarly, we say that a brane $L$ is compact along the second factor if its
 projection to $V_2$ is compact, or equivalently, the closure $\ol L$
lies in $\ol T^* V_2\simeq \ol V_1\times V_2$.
We can define a Fukaya $A_\infty$-structure on the collection 
of branes in $V_1\times V_2$
that are compact along the second factor. 
We write $F(T^*V_2)_\kappa$ for the triangulated envelope of the Fukaya category
of such branes. 
Consider the full subcategory  $Sh_c(V_2)_\kappa \subset Sh_c(V_2)$ 
of compactly supported objects. In parallel with the above discussion,
we have 
a quasi-embedding
$$
\mu_{V_2}: Sh_c(V_2)_\kappa \hra F(T^*V_2)_\kappa
$$
that leads to a right module
$$
\tilde \pi_{V_2}(L) : Sh_c(V_2)^{op}_\kappa \to \Ch
\qquad
\tilde \pi_{V_2}(L)(\CF) = \hom_{F(V_1\times V_2)_{pre}}(\mu_{V_2}(\CF), L)
$$
that is quasi-represented by an object of $Sh_c(V_2)$ which we denote by $\pi_{V_2}(L)$.
Fix a conical (with respect to the first factor) Lagrangian $\Lambda\hra T^*V_2$ such that the part of the boundary
$L^\infty$ that lies in $V^\infty_1\times V_2$ in fact lies in the boundary $\Lambda^\infty$.
Then for any stratification $\CS=\{S_\alpha\}$ 
of $V_2$ such that $\Lambda \subset \Lambda_\CS = \cup_\alpha T^*_{S_\alpha} V_2$,
the object $ \pi_{V_2}(L)$ lies in $Sh_{\CS}(V_2)$.

\subsubsection{Dilation invariance}

Recall that for $X$ a compact manifold, the natural dilation $\R^+$-action on $T^*X$ 
is not a symplectomorphism,
but as explained immediately after Theorem~\ref{thm equiv}, any brane $L\hra T^*X$  is quasi-isomorphic to its dilations.

The target $V_1\times V_2$ admits
two commuting $\R^+$-dilation actions which we denote by 
$\alpha_1^t$, $\alpha_2^t$,
for any $t\in \R^+$.
In general, given a brane $L\hra V_1\times V_2$, its dilations along either factor
will not define quasi-isomorphic sheaves on
the corresponding factor.

\begin{ex}
Suppose $V_1=  \R_x$, $V_2=\R_\xi$, and let $L\hra \R_x\times \R_\xi$ be a brane
with underlying Lagrangian $L = \{\xi= 1\}$.
Then $\pi_{V_1}(L)$ is quasi-isomorphic to its dilations,
but  $\pi_{V_2}(L)$ is not.
\end{ex}

\begin{prop}\label{prop dilate}
There are quasi-isomorphisms
$$
\pi_{V_1}(\alpha^t_2(L))\simeq \pi_{V_1}(L)
\qquad
\pi_{V_2}(\alpha^t_1(L))\simeq \pi_{V_2}(L)
$$
$$
\pi_{V_1}(\alpha^t_1(L))\simeq \alpha^t_{1*}(\pi_{V_1}(L))
\qquad
\pi_{V_2}(\alpha^t_2(L))\simeq \alpha^t_{2*}(\pi_{V_2}(L))
$$
\end{prop}

\begin{proof}
We prove the left hand column of assertions, the right hand column is the same.

For the first assertion, observe that the microlocalization $\mu_{V_1}$ is independent
of the dilation $\alpha^t_2$. To be precise, rather than scale the module $L$,
we can scale the image of $\mu_{V_1}$ by the inverse. 
But such scalings lead to quasi-isomorphic
calculations among standard branes.

For the second assertion, observe that the linear diffeomorphism $\alpha^t_1$
of the base $V_1$ induces the scaling $\alpha_1^t\circ \alpha_2^{t^{-1}}$
of its cotangent $T^*V_1\simeq V_1\times V_2$. Thus we have a
quasi-isomorphism
$$
\pi_{V_1}(\alpha_1^t \alpha_2^{t^{-1}}(L))\simeq \alpha^t_{1*}(\pi_{V_1}(L))
$$
since all of our constructions are invariant under linear diffeomorphisms.
On the other hand, by the first assertion applied to $\alpha_1^t(L)$, 
we have a quasi-isomorphism
$$
\pi_{V_1}(\alpha^{t^{-1}}_2\alpha_1^t(L))\simeq \pi_{V_1}(\alpha_1^t(L))
$$
Combining the above quasi-isomorphisms proves the second assertion.
\end{proof}


\subsection{Balanced branes}
Our main theorem will not apply to all branes $L\hra V_1\times V_2$,
but rather those satisfying a technical assumption which  we explain here.
(The main theorem fails without it.)

\medskip

Recall that the compactification 
$\ol V_1\times \ol V_2$ has a  codimension one boundary 
$B = (V_1 \times V_2^\infty) \coprod (V_1^\infty\times V_2)$ and 
a codimension two corner
$C=V_1^\infty \times V_2^\infty$.

Given a brane $L\hra V_1\times V_2$, let $L_C = \ol L \cap C\hra L^\infty$ be the intersection of the closure $\ol L$
with the corner $C$. Consider the cone 
$$
Cone(L_C) =\{(v_1, v_2)\in (\ol V_1\setminus\{0\})\times (\ol V_2\setminus\{0\})  | ([v_1] , [v_2]) \in L_C\}
\hra (\ol V_1\setminus\{0\})\times (\ol V_2\setminus\{0\}) 
$$
of nonzero elements whose projectivizations lie in $L_C$. 
By definition, it is invariant under the 
two commuting dilations $\alpha_1^t$, $\alpha_2^t$.

For $\delta_1, \delta_2\in \R^+$, consider the dilated brane 
$L(\delta_1, \delta_2) = \alpha_1^{\delta_1}\alpha_2^{\delta_2}L$.
We would expect that as $\delta_1, \delta_2 \to 0$, the dilated brane $L(\delta_1, \delta_2)$
would consist of two parts: a limit that collects along the closure
of the axes $\ol V_1\times\{0\}$, $\{0\}\times V_2$,
and the conical trace $Cone(L_C)$. In general, the situation is much more complicated.

By a balanced Lagrangian brane $L\hra V_1\times V_2$, we mean a
Lagrangian brane satisfying the following additional hypothesis:
the intersection $L_C$ of its closure $\ol L$ with the corner $C$ is
of the expected dimension:
$$
\dim L_C = \dim L - 2.
$$

This implies that for every neighborhood $\CN_{axes}$ of the closure of the axes 
$\ol V_1\times\{0\}$, $\{0\}\times V_2$, and every neighborhood $\CN_{cone}$
of the cone $Cone(L_C)$, there exists 
$\delta_1, \delta_2\in \R^+$ such that the dilated brane $L(\delta_1, \delta_2)$
lies in the union of the neighborhoods:
$$
L(\delta_1, \delta_2) \subset \CN_{axes} \cup \CN_{cone}.
$$

\begin{ex}
Suppose $V_1=  \R_x$, $V_2=\R_\xi$, and let $L\hra \R_x\times \R_\xi$ be a Lagrangian brane.
Then $L$ is symmetric if and only if the closure $\ol L\hra \ol \R_x \times \ol \R_\xi$
is disjoint from the four corners $\{(\pm\infty, \pm\infty)\}\subset \ol \R_x \times \ol \R_\xi$.
For instance, the Lagrangian $\{\xi = 1/x\}$ can underlie a balanced brane, 
but the Lagrangian $\{x=\xi\}$ can not underlie a balanced brane.
\end{ex}

The above example explains our use of the term balanced. The 
intersection of the closure $\ol L$ with the corners $C$ is unstable, 
and threatens to teeter over in the direction of either factor.


%



\subsection{Main theorem}
This section contains the main technical result of this paper.
In order to help the reader follow our arguments, we have
isolated the case when $\dim V_1 =1$. Furthermore, the arguments of the general case
are best understood as a product of copies of the dimension one case.


\subsubsection{Case of dimension one}
The fearless reader could skip this section, and continue in the next section
with the general case.
But many of the intricacies of the general case already appear here.
Furthermore, the constructions of the general case can be understood
as a product of constructions described here.
With this in mind, our aim here is not to give the most concrete proof possible, 
but rather to argue in parallel with what will be required for the general case.

We will write $\R_x$ to denote $V_1$ with coordinate $x$,
and $\R_\xi$ to denote $V_2 = V_1^*$ with dual coordinate $\xi$.
Recall that a brane $L\hra \R_x\times \R_\xi$ is balanced if its closure 
$\ol L \hra \ol \R_x\times\ol \R_\xi$
is disjoint from the four corners $\{(\pm\infty, \pm\infty)\}\subset \ol \R_x \times \ol \R_\xi$.

\begin{thm}
Let $L\hra \R_x\times \R_\xi$ be a balanced brane. Then there are quasi-isomorphisms
$$
(\Upsilon(\pi_{\R_x}(L)))^\w \simeq \Upsilon(\pi_{\R_\xi}(L))
\qquad
(\Upsilon(\pi_{\R_\xi}(L))^\vee \simeq \Upsilon(\pi_{\R_x}(L)).
$$
\end{thm}

\begin{proof}
We prove the first identity, the second follows immediately by applying the inverse
Fourier transform to the first.

Consider an object $\CF\in Sh_c(\R_x/\R^+)$, and its Fourier transform $\CF^\w\in Sh_c(\R_\xi/\R^+)$.
Recall that for any open convex cone $u:U\hra \R_\xi$, with closed polar cone $v:U^\circ\hra \R_x$
with interior $\interiorsymbol (v):\interiorsymbol (U^\circ)\hra \R_x$,
we have a quasi-isomorphism
$$
\hom_{Sh_c(\R_\xi)}(u_!\C_U,\CF^\w) \simeq
\hom_{Sh_c(\R_x)}(\interiorsymbol (v)_*\C_{\interiorsymbol (U^\circ)}, \CF).
$$
Furthermore, for the inclusion of 
open convex cones $U_0\hra U_1\hra \R_\xi$,
the above quasi-isomorphisms fit into a commutative (at the level of cohomology) square.
The resulting compatible collection of quasi-isomorphisms characterizes $\CF^\w$.

Thus to prove the first assertion, it suffices to establish the formula 
\begin{equation}\label{key formula dim = 1}\tag{\dag}
\hom_{Sh_c(\R_\xi)}(u_!\C_U,\Upsilon(\pi_{\R_\xi}(L))) \simeq
\hom_{Sh_c(\R_x)}(\interiorsymbol (v)_*\C_{\interiorsymbol (U^\circ)}, \Upsilon(\pi_{\R_x}(L))).
\end{equation}
 compatibly for all open convex cones.

In dimension one, it is possible to list all of the open convex cones
$$
 \begin{array}{cccccc}
U=\R^+_\xi
&&
U=\R^-_\xi
&&
U=\R_\xi\\

U^\circ=\ol \R^+_x
&&
U^\circ=\ol \R^-_x
&&
U^\circ=\{0\}
\end{array}
$$

  We will establish formula~(\ref{key formula dim = 1}) for $U=\R^+_\xi$, $U^\circ=\ol \R^+_x$.
 We leave it to the reader to modify the arguments for the other cases,
 and to check that the constructions are compatible with inclusions.

\medskip

Thus our aim is to show that 
there is a quasi-isomorphism
\begin{equation}\label{key formula specialized dim = 1}\tag{\ddag}
\hom_{Sh_c(\R_\xi)}(u_!\C_{\R^+_\xi},\Upsilon(\pi_{\R_\xi}(L))) \simeq
\hom_{Sh_c(\R_x)}(u_*\C_{\R^+_x}, \Upsilon(\pi_{\R_x}(L))).
\end{equation}

Our strategy will be to construct a brane in $\R_x\times\R_\xi$ such that
both sides of formula~(\ref{key formula  specialized dim = 1}) are quasi-isomorphic to its Floer pairing
with a dilation of $L$. 

Fix a pair $\varepsilon_x, \vareps_\xi\in \R$
(soon to be
specialized to the case $\vareps_x < 0$, $\vareps_\xi > 0$), 
and consider the open subsets
$$
q({\vareps_x}): Q({\vareps_x}) =\{x\in \R_x | x >\vareps_x\} \hra \R_x
\qquad
q({\vareps_\xi}): Q({\vareps_\xi}) =\{\xi\in \R_\xi | \xi>\vareps_\xi\} \hra \R_\xi
$$
and the Lagrangian
$$
P({\vareps_x,\vareps_\xi}) = \{(x, \xi)\in \R_x\times\R_\xi | x >\vareps_x , (x-\vareps_x)(\xi-\vareps_\xi) = 1\} \hra \R_x\times\R_\xi.
$$
We equip $P({\vareps_x,\vareps_\xi})$ with the brane structure
coming from its identification with the standard brane $L_{Q({\vareps_x})*}\hra T^*\R_x$,
or equivalently its identification with the costandard brane $L_{Q({\vareps_\xi})!}\hra T^*\R_\xi$.
Note that the boundary of ${P({\vareps_x,\vareps_\xi})}$ inside of $\ol \R_x\times\ol \R_\xi$
consists of the two points $(\vareps_x, +\infty)$
and $(+\infty, \vareps_\xi)$.

\medskip

For any $\delta_x,\delta_\xi\in\R^+$,
consider the dilated brane
$$
L(\delta_x,\delta_\xi) =\alpha_x({\delta_x})\alpha_\xi({\delta_\xi})L \hra \R_x\times\R_\xi.
$$
By assumption, $L$ is balanced, so the boundary of $L$
 inside of $\ol \R_x\times\ol \R_\xi$
 is disjoint from the corners
 $\{(\pm\infty, \pm\infty)\}$.
Thus for small $\delta_x,\delta_\xi\in\R^+$,
the boundary of $L(\delta_x,\delta_\xi)$
is arbitrarily close to the points $\{(\pm \infty, 0), (0,\pm \infty)\}$.
 In particular, the boundary of $L(\delta_1, \delta_2)$ 
 does not intersect the boundary of $P(\vareps_x,\vareps_\xi)$.
 Thus 
  for sufficiently small $\delta_x,\delta_\xi \in \R^+$,
it makes sense to 
consider the Floer complex
\begin{equation}\label{key floer complex dim = 1}\tag{Fl}
\hom_{F(\R_x\times \R_\xi)}(P({\vareps_x,\vareps_\xi}) ,L(\delta_x,\delta_\xi) )
\end{equation}

We claim that for fixed $\vareps_x < 0$, $\vareps_\xi > 0$, 
and sufficiently small $\delta_x,\delta_\xi \in \R^+$,
both sides of formula~(\ref{key formula specialized dim = 1})
are quasi-isomorphic to the Floer complex~(\ref{key floer complex dim = 1}).
We will first explain why the right hand side of~(\ref{key formula  specialized dim = 1})
is quasi-isomorphic to~(\ref{key floer complex dim = 1}), and
then give the parallel arguments for the left hand side.

\medskip

Thus our immediate aim is to show that  for fixed $\vareps_x < 0$, $\vareps_\xi > 0$, 
and sufficiently small $\delta_x,\delta_\xi \in \R^+$,
there is a quasi-isomorphism
\begin{equation}\label{identify rhs n=1}\tag{rhs}
\hom_{F(\R_x\times \R_\xi)}(P({\vareps_x,\vareps_\xi}) ,L(\delta_x,\delta_\xi) )
\simeq
\hom_{Sh_c(\R_x)}(u_*\C_{\R^+_x}, \Upsilon(\pi_{\R_x}(L))).
\end{equation}

Let us unpack the right hand side of the sought after identity~(\ref{identify rhs  n=1}).
To that end, for large $r_x\in\R^+$, consider the intervals
$$
\xymatrix{
[\vareps_x, r) \ar@{^(->}[r]^-{c} & [\vareps_x, r] \ar@{^(->}[r]^-{d} & \R_x,
}
$$
and define the pushforward
$$ 
\CR(\vareps_x, r_x) = d_*c_!\C_{[\vareps_x, r)}\in Sh_c(\R_x).
$$ 
Consider the corresponding
  brane $R{(\varepsilon_x, r_x)}=\mu_{\R_x}(\CR(\vareps_x, r_x))$
  with underlying Lagrangian 
  $$
  R{(\varepsilon_x , r_x)}= \{(x,\xi)\in \R_x\times\R_\xi| x\in (\vareps_x, r_x), \xi(x-\vareps_x)(r_x-x) = 1\}
  \hra \R_x\times\R_\xi.
  $$

\medskip

As long as $\delta_\xi\in \R^+$ is sufficiently small,
by Lemma~\ref{lem sections calc}, Proposition~\ref{prop dilate}, standard adjunctions,
and the definition of $\pi_{\R_x}$,
 we have quasi-isomorphisms
\begin{eqnarray*}
\hom_{Sh_c(\R_x)}(u_*\C_{\R^+_x}, \Upsilon(\pi_{\R_x}(L)))
& \simeq & 
\hom_{Sh_c(E)}(\CR{(\varepsilon_x, r_x)},\pi_{\R_x}(L(\delta_x, \delta_\xi)))\\
& \simeq &
\hom_{F(\R_x\times \R_\xi)}(R{(\varepsilon_x, r_x)}, L(\delta_x, \delta_\xi)).
\end{eqnarray*}

Thus to establish the sought after identity~(\ref{identify rhs  n=1}),
it suffices to establish a quasi-isomorphism
\begin{equation}\label{identify rhs reinterpret n=1}\tag{$\star_x$}
\hom_{F(\R_x\times \R_\xi)}(P({\vareps_x,\vareps_\xi}) ,  L(\delta_x,\delta_\xi))
\simeq
\hom_{F(\R_x\times \R_\xi)}(R{(\varepsilon_x, r_x)}, L(\delta_x, \delta_\xi)).
\end{equation}

Again, since $L$ is balanced,
for small $\delta_x,\delta_\xi\in\R^+$,
the boundary of $L(\delta_x,\delta_\xi)$
 inside of $\ol \R_x\times\ol \R_\xi$
is arbitrarily close to the points $\{(\pm \infty, 0), (0,\pm \infty)\}$.
  Thus we can find a $[0,1]$-family of branes 
  $
  \fP_t\hra \R_x\times\R_\xi
  $
  such that 
  $$
   \fP_0({\vareps_x,\vareps_\xi}) =  P({\vareps_x,\vareps_\xi})
   \qquad
    \fP_1({\vareps_x,\vareps_\xi}) = R({\vareps_x, r_x}) 
    $$
    and $\fP_t$ is non-characteristic with respect to $L(\delta_x,\delta_\xi)$.
    In other words, the boundary of $\fP_t$ is disjoint from the boundary of 
     $L(\delta_x,\delta_\xi)$, for all time $t\in [0,1]$.

Thus by Proposition~\ref{prop brane non-char in general}, we have the sought after identity~(\ref{identify rhs reinterpret n=1}),
and in turn the identity~(\ref{identify rhs n=1}).

\medskip

Our next aim is to give parallel arguments explaining
why the left hand side of~(\ref{key formula  specialized dim = 1})
is quasi-isomorphic to~(\ref{key floer complex dim = 1}),
So we need to verify that  for fixed $\vareps_x < 0$, $\vareps_\xi > 0$, 
and sufficiently small $\delta_x,\delta_\xi \in \R^+$,
there is a quasi-isomorphism
\begin{equation}\label{identify lhs n=1}\tag{lhs}
\hom_{F(\R_x\times \R_\xi)}(P({\vareps_x,\vareps_\xi}) ,L(\delta_x,\delta_\xi) )
\simeq
\hom_{Sh_c(\R_\xi)}(u_!\C_{\R^+_\xi},\Upsilon(\pi_{\R_\xi}(L))).
\end{equation}

Let us unpack the right hand side of the sought after identity~(\ref{identify lhs n=1}).
To that end, for large $r_\xi\in\R^+$, consider the interval
$
c:(\vareps_\xi, r_\xi) \hra  \R_\xi
$
and define the pushforward
$$ 
\CT{(\vareps_\xi, r_\xi)} = c_!\C_{(\vareps_\xi, r_\xi)}\in Sh_c(\R_\xi).
$$ 
Consider the corresponding
costandard  brane $T{(\varepsilon_\xi, r_\xi)}=\mu_{\R_\xi}(\CT{(\vareps_\xi, r_\xi)})$
  with underlying Lagrangian 
  $$
 T{( \varepsilon_\xi, r_\xi)}
 = \{(x,\xi)\in \R_x\times\R_\xi| \xi\in (\vareps_\xi, r_\xi), x(\xi-\vareps_\xi)(r_\xi-\xi) = 
  2\xi - r_\xi - \vareps_\xi\}
  \hra \R_x\times\R_\xi.
  $$

\medskip

As long as $\delta_x\in \R^+$ is sufficiently small,
by Lemma~\ref{lem sections calc}, Proposition~\ref{prop dilate}, standard adjunctions,
and the definition of $\pi_{\R_\xi}$,
 we have quasi-isomorphisms
\begin{eqnarray*}
\hom_{Sh_c(\R_\xi)}(u_!\C_{\R^+_\xi}, \Upsilon(\pi_{\R_\xi}(L)))
& \simeq & 
\hom_{Sh_c(E)}(\CT{( \varepsilon_\xi, r_\xi)} ,\pi_{\R_\xi}(L(\delta_x, \delta_\xi)))\\
& \simeq &
\hom_{F(\R_x\times \R_\xi)}( T{( \varepsilon_\xi, r_\xi)}, L(\delta_x, \delta_\xi)).
\end{eqnarray*}

Thus to establish the sought after identity~(\ref{identify lhs  n=1}),
it suffices to establish a quasi-isomorphism
\begin{equation}\label{identify lhs reinterpret n=1}\tag{$\star_\xi$}
\hom_{F(\R_x\times \R_\xi)}(P({\vareps_x,\vareps_\xi}) ,  L(\delta_x,\delta_\xi))
\simeq
\hom_{F(\R_x\times \R_\xi)}( T{( \varepsilon_\xi, r_\xi)}, L(\delta_x, \delta_\xi)).
\end{equation}

Again, since $L$ is balanced,
for small $\delta_x,\delta_\xi\in\R^+$,
the boundary of $L(\delta_x,\delta_\xi)$
 inside of $\ol \R_x\times\ol \R_\xi$
is arbitrarily close to the points $\{(\pm \infty, 0), (0,\pm \infty)\}$.
  Thus we can find a $[0,1]$-family of branes 
  $
  \fP_t\hra \R_x\times\R_\xi
  $
  such that 
  $$
   \fP_0({\vareps_x,\vareps_\xi}) =  P({\vareps_x,\vareps_\xi})
   \qquad
    \fP_1({\vareps_x,\vareps_\xi}) =T{( \varepsilon_\xi, r_\xi)}
    $$
    and $\fP_t$ is non-characteristic with respect to $L(\delta_x,\delta_\xi)$.
    In other words, the boundary of $\fP_t$ is disjoint from the boundary of 
     $L(\delta_x,\delta_\xi)$, for all time $t\in [0,1]$.

Thus by Proposition~\ref{prop brane non-char in general}, 
we have the sought after identity~(\ref{identify lhs reinterpret n=1}),
and in turn the identity~(\ref{identify lhs n=1}).

\medskip

Putting the preceding together, we have identified the left~(\ref{identify lhs n=1}) and right hand
~(\ref{identify rhs n=1}) sides
of  formula~(\ref{key formula specialized dim = 1}) with the 
Floer complex~(\ref{key floer complex dim = 1}).
As mentioned above, we leave to the reader to modify the arguments
to establish formula~(\ref{key formula dim = 1}) in the other cases,
and to check its compatibility with inclusions.
This establishes the first assertion of the theorem.
\end{proof}


\subsubsection{General case}

Now we arrive the at main technical result of this paper in Theorem~\ref{main thm}.
We hope the reader has followed the arguments of the preceding section
in the case of dimension one. 
Here we return to the general setting
of an arbitrary real finite dimensional vector space $V_1$ with dual $V_2$.

\begin{thm}\label{main thm}
Let $L\hra V_1\times V_2$ be a balanced brane. Then there are quasi-isomorphisms
$$
(\Upsilon(\pi_{V_1}(L)))^\w \simeq \Upsilon(\pi_{V_2}(L))
\qquad
(\Upsilon(\pi_{V_2}(L))^\vee \simeq \Upsilon(\pi_{V_1}(L)).
$$
\end{thm}

Before giving the proof, let us record the following special case
which we will apply in the context of Springer theory.

\begin{cor}
Suppose $L$ is conical along the second factor in the sense that $\alpha_2^t L = L$, 
for all $t\in \R^+$. Then we have
$
(\Upsilon(\pi_{V_1}(L))^\w \simeq \pi_{V_2}(L).
$
\end{cor}

\begin{proof}
Follows from the theorem, Proposition~\ref{prop dilate}, and the fact that $\Upsilon$ is the identity
functor on conic objects.
\end{proof}

\begin{proof}
[Proof of Theorem~\ref{main thm}]
We prove the first identity, the second follows immediately by applying the inverse
Fourier transform to the first.

Consider an object $\CF\in Sh_c(V_1/\R^+)$, and its Fourier transform $\CF^\w\in Sh_c(V_2/\R^+)$.
Recall that for any convex open cone $u:U\hra V_2$, with closed polar cone $v:U^\circ\hra V_1$
with interior $\interiorsymbol (v):\interiorsymbol (U^\circ)\hra V_1$,
we have a quasi-isomorphism
$$
\hom_{Sh_c(V_2)}(u_!\C_U,\CF^\w) \simeq
\hom_{Sh_c(V_1)}(\interiorsymbol (v)_*\C_{\interiorsymbol (U^\circ)}, \CF).
$$
Furthermore, for the inclusion of 
open convex cones $U_0\hra U_1\hra V_2$,
the above quasi-isomorphisms fit into a commutative (at the level of cohomology) square.
The resulting compatible collection of quasi-isomorphisms characterizes $\CF^\w$.

Thus to prove the first assertion, it suffices to establish the formula 
\begin{equation}\label{key formula}\tag{\dag}
\hom_{Sh_c(V_2)}(u_!\C_U,\Upsilon(\pi_{V_2}(L))) \simeq
\hom_{Sh_c(V_1)}(\interiorsymbol (v)_*\C_{\interiorsymbol (U^\circ)}, \Upsilon(\pi_{V_1}(L))).
\end{equation}
 compatibly for all open convex cones.

In fact, it suffices to 
 establish formula~(\ref{key formula}) for any collection of  open cones as long as they 
 generate the conic topology. For technical convenience,
 we will focus on open cones $u:U\hra V_2$
which become identified $U\simeq Q^n_\xi$
 with the standard open quadrant 
$$
q:Q^n_\xi = \{ (\xi_1, \ldots, \xi_n)\in \R_\xi^n | \xi_i > 0\} \hra \R_\xi^n
$$
under some linear isomorphism $V_2\simeq \R_\xi^n$.
We refer to such open cones as open quadrant cones.
  Note that the interior of the closed polar cone of the standard open quadrant $q:Q^n_\xi\hra \R^n_\xi$ 
  is nothing other than
  the standard open quadrant $q:Q^n_x\hra \R^n_x$.
  The collection of open quadrant cones, together with $V_2$ itself,
generate the conic topology.

It will be useful to specialize further to a particular collection of open quadrant cones.
For any $\delta_x,\delta_\xi\in\R^+$,
consider the dilated brane
$$
L(\delta_x,\delta_\xi) =\alpha_x({\delta_x})\alpha_\xi({\delta_\xi})L \hra V_1\times V_2.
$$
Recall that since $L$ is balanced, there exists a biconic Lagrangian 
$\Lambda\hra V_1\times V_2$ such that
for every neighborhood $\CN_{axes}$ of the closure of the axes 
$\ol V_1\times\{0\}$, $\{0\}\times V_2$, and every neighborhood $\CN_{cone}$
of $\Lambda$, there exists 
$\delta_1, \delta_2\in \R^+$ such that the dilated brane $L(\delta_1, \delta_2)$
lies in the union of the neighborhoods
$$
L(\delta_1, \delta_2) \subset \CN_{axes} \cup \CN_{cone}.
$$

Given an open quadrant cone $u:U\hra V_2$, consider the biconic Lagrangian 
$$\Lambda_U\hra T^*\R^n_\xi \simeq
V_1\times V_2$$ 
obtained by taking the union of the conormals to the facets
of the boundary $\del U \hra \R^n_\xi$.
We will focus on open quadrant cones $u:U\hra V_2$ such that
$\Lambda_U$ is disjoint from $\Lambda$ away from $\CN_{axes}$. By a dimension count,
one can check that this is a generic condition. 
Hence 
the collection of all such open quadrant cones, together with $V_2$ itself,
generate the conic topology.

Now without loss of generality, 
 to establish formula~(\ref{key formula}) for a given generic quadrant cone $u:U\hra V_2$, 
 it suffices to choose an identification $V_2\simeq \R^n_\xi$, and to
establish (\ref{key formula}) for the standard open quadrant $q:Q^n_\xi\hra \R^n_\xi$.
This case of  formula~(\ref{key formula}) will occupy the remainder of our arguments. 
 We leave it to the reader to handle the case $U=\R^n_\xi$ itself,
  and to check that our constructions are compatible with inclusions.

\medskip

Thus our aim is to show that 
there is a quasi-isomorphism
\begin{equation}\label{key formula specialized}\tag{\ddag}
\hom_{Sh_c(\R^n_\xi)}(q_!\C_{Q^n_\xi},\Upsilon(\pi_{V_2}(L))) \simeq
\hom_{Sh_c(\R^n_x)}(q_*\C_{Q^n_x}, \Upsilon(\pi_{V_1}(L))).
\end{equation}

Our strategy will be to construct a brane in $\R^n_x\times \R^n_\xi$ such that
both sides of formula~(\ref{key formula  specialized}) are quasi-isomorphic to its Floer pairing
with a dilation of $L$. 

Fix a pair $\varepsilon_x, \vareps_\xi\in \R$
(to be
specialized to the case $\vareps_x < 0$, $\vareps_\xi > 0$ momentarily), 
and consider the open subsets
$$
q({\vareps_x}): Q_x^n({\vareps_x}) =\{x\in \R^n_x | x_i >\vareps_x, \mbox{ for $i=1,\ldots, n$}\} 
\hra \R^n_x
$$
$$
q({\vareps_\xi}): Q_\xi^n({\vareps_\xi}) =\{\xi\in \R^n_\xi | \xi_i>\vareps_\xi,  \mbox{ for $i=1,\ldots, n$}\} 
\hra \R^n_\xi
$$
By definition, we have $Q_x^n(0) = Q_x^n$ and $Q_\xi^n(0) = Q_\xi^n$.

Consider inside of $\R^n_x\times \R^n_\xi$ the Lagrangian
$$
P({\vareps_x,\vareps_\xi}) = \{(x, \xi)\in \R^n_x\times \R^n_\xi | 
x_i >\vareps_x , (x_i-\vareps_x)(\xi_i-\vareps_\xi) = 1, 
 \mbox{ for $i=1,\ldots, n$}\}
$$
equipped with the brane structure
coming from its identification with the standard brane 
$$L_{Q^n_x({\vareps_x})*}\hra T^*\R^n_x,$$
or equivalently, its identification with the costandard brane 
$$
L_{Q^n_\xi({\vareps_\xi})!}\hra T^* \R^n_\xi.
$$

By construction, 
for fixed $\vareps_x < 0$, $\vareps_\xi > 0$, 
and sufficiently small 
$\delta_x,\delta_\xi\in\R^+$,
the boundary of the dilated brane $L(\delta_1, \delta_2)$ 
 does not intersect the boundary of $P(\vareps_x,\vareps_\xi)$.
 Thus it makes sense to 
consider the Floer complex
\begin{equation}\label{key floer complex}\tag{Fl}
\hom_{F(\R^n_x\times  \R^n_\xi)}(P({\vareps_x,\vareps_\xi}) ,L(\delta_x,\delta_\xi) ),
\mbox{ for sufficiently small $\delta_x,\delta_\xi \in \R^+$}.
\end{equation}

We claim that 
both sides of formula~(\ref{key formula specialized})
are quasi-isomorphic to the Floer complex~(\ref{key floer complex}).
We will first explain why the right hand side of~(\ref{key formula  specialized})
is quasi-isomorphic to~(\ref{key floer complex}), and
then give the parallel arguments for the left hand side.

\medskip

Thus our immediate aim is to show that  for fixed $\vareps_x < 0$, $\vareps_\xi > 0$, 
and sufficiently small $\delta_x,\delta_\xi \in \R^+$,
there is a quasi-isomorphism
\begin{equation}\label{identify rhs}\tag{rhs}
\hom_{F(\R^n_x\times  \R^n_\xi)}(P({\vareps_x,\vareps_\xi}) ,L(\delta_x,\delta_\xi) )
\simeq
\hom_{Sh_c(\R^n_x)}(q_*\C_{Q^n_x}, \Upsilon(\pi_{V_1}(L))).
\end{equation}

Let us unpack the right hand side of the sought after identity~(\ref{identify rhs}).
To that end, let us introduce the translated variables $\hat x_i = x_i - \vareps_x$,
for $i=1, \ldots, n$.
For large $r_x\in\R^+$, consider the truncations
$$
\xymatrix{
Q^n_x({\vareps_x})\cap \{|\hat x|^2<r_x\}
 \ar@{^(->}[r]^-{c} & 
 Q^n_x({\vareps_x})\cap \{ |\hat x|^2\leq r_x\}
\ar@{^(->}[r]^-{d} & \R^n_x,
}
$$
and define the pushforward
$$ 
\CR{({\vareps_x}, r_x)} = d_*c_!\C_{Q^n_x({\vareps_x})\cap  \{|\hat x|^2<r_x\}}
\in Sh_c(\R^n_x).
$$ 
Consider the corresponding
  brane $R{(\varepsilon_x, r_x)}=\mu_{\R^n_x}(\CR{({\vareps_x}, r_x)} )$
  with underlying Lagrangian 
  $$
R{(\varepsilon_x, r_x)}
= \{(x,\xi)\in \R^n_x\times \R^n_\xi | 
\hat x_i> 0, \xi\hat x_i(r_x-|\hat x|^2) = 1, |\hat x|^2 < r_x
\}
  \hra \R^n_x\times \R^n_\xi.
  $$

\medskip

As long as $\delta_\xi\in \R^+$ is sufficiently small,
by Lemma~\ref{lem sections calc}, Proposition~\ref{prop dilate}, standard adjunctions,
and the definition of $\pi_{V_1}$,
 we have quasi-isomorphisms
\begin{eqnarray*}
\hom_{Sh_c(\R^n_x)}(q_*\C_{Q^n_x}, \Upsilon(\pi_{\R^n_x}(L)))
& \simeq & 
\hom_{Sh_c(E)}(\CR{({\vareps_x}, r_x)} ,\pi_{V_1}(L(\delta_x, \delta_\xi)))\\
& \simeq &
\hom_{F(\R^n_x\times  \R^n_\xi)}(R{(\varepsilon_x, r_x)}, L(\delta_x, \delta_\xi)).
\end{eqnarray*}

Thus to establish the sought after identity~(\ref{identify rhs  n=1}),
it suffices to establish a quasi-isomorphism
\begin{equation}\label{identify rhs reinterpret}\tag{$\star_x$}
\hom_{F(\R^n_x\times  \R^n_\xi)}(P({\vareps_x,\vareps_\xi}) ,  L(\delta_x,\delta_\xi))
\simeq
\hom_{F(\R^n_x\times  \R^n_\xi)}(R{(\varepsilon_x, r_x)}, L(\delta_x, \delta_\xi)).
\end{equation}

By construction,
we can find a $[0,1]$-family of branes 
  $
  \fP_t\hra \R^n_x\times \R^n_\xi
  $
  such that 
  $$
   \fP_0({\vareps_x,\vareps_\xi}) =  P({\vareps_x,\vareps_\xi})
   \qquad
    \fP_1({\vareps_x,\vareps_\xi}) = R{(\varepsilon_x, r_x)}
    $$
    and $\fP_t$ is non-characteristic with respect to $L(\delta_x,\delta_\xi)$.
    In other words, the boundary of $\fP_t$ is disjoint from the boundary of 
     $L(\delta_x,\delta_\xi)$, for all time $t\in [0,1]$.
     To explicitly define $\fP_t$, one can exploit that an open quadrant
     cone is a product of one-dimensional cones.

Thus by Proposition~\ref{prop brane non-char in general}, 
we have the sought after identity~(\ref{identify rhs reinterpret}),
and in turn the identity~(\ref{identify rhs}).

\medskip

Next we give parallel arguments explaining
why the left hand side of~(\ref{key formula  specialized})
is quasi-isomorphic to~(\ref{key floer complex}).
So we need to verify that  for fixed $\vareps_x < 0$, $\vareps_\xi > 0$, 
and sufficiently small $\delta_x,\delta_\xi \in \R^+$,
there is a quasi-isomorphism
\begin{equation}\label{identify lhs}\tag{lhs}
\hom_{F(\R^n_x\times  \R^n_\xi)}(P({\vareps_x,\vareps_\xi}) ,L(\delta_x,\delta_\xi) )
\simeq
\hom_{Sh_c( \R^n_\xi)}(u_!\C_{\R^+_\xi},\Upsilon(\pi_{2}(L))).
\end{equation}

Let us unpack the right hand side of the sought after identity~(\ref{identify lhs}).
To that end, 
let us introduce the translated variables $\hat \xi_i = \xi_i - \vareps_\xi$,
for $i=1, \ldots, n$.
For large $r_\xi\in\R^+$, consider the truncation
$$
c:Q^n_\xi(\vareps_\xi)\cap \{ |\hat\xi|^2< r_\xi\} \hra   \R^n_\xi
$$
and define the pushforward
$$ 
\CT{(\vareps_\xi, r_\xi)} = c_!\C_{Q^n_\xi(\vareps_\xi)\cap \{ |\hat \xi|^2< r_\xi\} }\in Sh_c( \R^n_\xi).
$$ 
Consider the corresponding
costandard  brane $T{(\vareps_\xi, r_\xi)} =\mu_{ \R^n_\xi}(\CT{(\vareps_\xi, r_\xi)} )$
  with underlying Lagrangian 
  $$
 T{(\vareps_\xi, r_\xi)}  = 
 \{(x,\xi)\in \R^n_x\times \R^n_\xi|
\hat \xi_i> 0,  \hat \xi_i x_i (r_\xi - |\hat \xi|^2)= r_\xi - |\hat \xi|^2 + 2\hat \xi_i^2,
 |\hat \xi|^2 < r_\xi\}
  $$

\medskip

As long as $\delta_x\in \R^+$ is sufficiently small,
by Lemma~\ref{lem sections calc}, Proposition~\ref{prop dilate}, standard adjunctions,
and the definition of $\pi_{V_2}$,
 we have quasi-isomorphisms
\begin{eqnarray*}
\hom_{Sh_c( \R^n_\xi)}(u_!\C_{\R^+_\xi}, \Upsilon(\pi_{2}(L)))
& \simeq & 
\hom_{Sh_c(E)}(\CT{( \vareps_\xi, r_\xi)} ,\pi_{V_2}(L(\delta_x, \delta_\xi)))\\
& \simeq &
\hom_{F(\R^n_x\times  \R^n_\xi)}( T{( \vareps_\xi, r_\xi)} , L(\delta_x, \delta_\xi)).
\end{eqnarray*}

Thus to establish the sought after identity~(\ref{identify lhs  n=1}),
it suffices to establish a quasi-isomorphism
\begin{equation}\label{identify lhs reinterpret}\tag{$\star_\xi$}
\hom_{F(\R^n_x\times  \R^n_\xi)}(P({\vareps_x,\vareps_\xi}) ,  L(\delta_x,\delta_\xi))
\simeq
\hom_{F(\R^n_x\times  \R^n_\xi)}(T(\vareps_\xi, r_\xi) , L(\delta_x, \delta_\xi)).
\end{equation}

By construction, we can find a $[0,1]$-family of branes 
  $
  \fP_t\hra \R^n_x\times \R^n_\xi
  $
  such that 
  $$
   \fP_0({\vareps_x,\vareps_\xi}) =  P({\vareps_x,\vareps_\xi})
   \qquad
    \fP_1({\vareps_x,\vareps_\xi}) =T( \vareps_\xi, r_\xi)   
     $$
    and $\fP_t$ is non-characteristic with respect to $L(\delta_x,\delta_\xi)$.
    In other words, the boundary of $\fP_t$ is disjoint from the boundary of 
     $L(\delta_x,\delta_\xi)$, for all time $t\in [0,1]$.
  To explicitly define $\fP_t$, one can exploit that an open quadrant
     cone is a product of one-dimensional cones.

Thus by Proposition~\ref{prop brane non-char in general}, 
we have the sought after identity~(\ref{identify lhs reinterpret}),
and in turn the identity~(\ref{identify lhs}).

\medskip

Putting the preceding together, we have identified the left~(\ref{identify lhs n=1}) and right hand
~(\ref{identify rhs n=1}) sides
of  formula~(\ref{key formula specialized dim = 1}) with the 
Floer complex~(\ref{key floer complex dim = 1}).
As mentioned above, we leave to the reader to modify the arguments
to establish formula~(\ref{key formula dim = 1}) in the other cases,
and to check its compatibility with inclusions.
This establishes the first assertion of the theorem, and consequently the second.
\end{proof}


\section{Springer theory}\label{sec springer}

In this section, we give an application of the preceding theory to branes
in the cotangent bundle of a Lie algebra.


\subsection{Springer theory via sheaves}
We begin with a brief review of the Springer theory
of Weyl group representations following Lusztig~\cite{L}, Borho-MacPherson~\cite{BoMac}, 
Ginzburg~\cite{Ginz}, and Hotta-Kashiwara~\cite{HK}. 
For a expanded summary, the reader could consult
the discussion in~\cite{Gr}.

\medskip

Let $G$ be a reductive complex algebraic group with Lie algebra $\g$
and Weyl group $W$.
Let $\CB$ be the flag variety of Borel subalgebras $\b\subset \g$.
For a Borel subalgebra $\b\subset \g$, let $\fu\subset \b$ be its unipotent radical,
and let $\fh=\b/\fu$ be the universal Cartan algebra.

\medskip

Let $\wt \g=\{(X,\b) | x\in \b\in \CB\}$ be the Grothendieck-Springer
space of pairs.
Let $\CN\subset \g$ be the nilpotent cone, and $\wt \CN=\{(x,\b) | x\in \b\cap \CN, \b\in \CB\}$ 
the Springer resolution. The following diagram summarizes some of the well-known relations among
these spaces:

$$
\xymatrix{
\CB & \ar[d]_-{\mu_\CN} \ar[l]_-{p} T^*\CB \simeq \wt \CN \ar[r]^-{\wt\imath} 
&  \wt \g  \ar[r]^-{\wt q} \ar[d]_{\mu_\g}
& \mathcal \fh \ar[d]_-{s}\\
&  \CN \ar[r]^{i} & \g \ar[r]^-{q} & \fh\inv W
}
$$

Here $p$, $\mu_\CN$, $\mu_\g$, and $s$ are the obvious projections, $\wt \imath$ and $i$
are the obvious
inclusions, $\wt q$ assigns 
to $(x,\b)$ the class of $x$ in $\b/\fu$, and $q$ is the affine adjoint quotient map.

\medskip

Let $\g_{rs}\subset \g$ denote the regular semisimple locus, and let 
$\fh_{r}\subset \fh$ denote the $W$-regular locus. The right-hand portion of the
above diagram restricts to the following Cartesian diagram whose vertical arrows
are $W$-torsors:

$$\xymatrix{
 \wt \g_{rs}  \ar[r]^-{\wt q} \ar[d]_{\mu_\g}
& \mathcal \fh_r \ar[d]_-{s}\\
 \g_{rs} \ar[r]^-{q} & \fh_r  \inv W
}
$$

Consider the constant perverse sheaves 
$$\C_{\wt\g}[\dim_\C\g]\in Perv(\wt\g)
\qquad
\C_{\wt \CN}[\dim_\C \CN] \in Perv(\wt\CN),
$$ 
and their pushforwards
$$
S_\g = \mu_{\g !} \C_{\wt \g}[\dim_\C \g]
\qquad
S_\CN = \mu_{\CN !} \C_{\wt \CN}[\dim_\C \CN]
$$
We refer to $S_\g$ as the global Springer sheaf, and $S_\CN$ as the nilpotent
Springer sheaf.

The map $\mu_{\g}$ is small and the map $\mu_{\CN}$ is semismall.
Thus the complexes $S_{\g}, S_\CN$ are in fact
perverse, and $S_{\g}$ is the middle extension
of the local system
$$
\CL_{rs}= \mu_{\g ! }\C_{\wt\g_{rs}}[\dim_\C\g] \in Perv(\g_{rs}).
$$

The Weyl group $W$ acts on $\CL_{rs}$ by deck transformations, and hence on $S_\g$
by the functoriality of the middle extension. The action identifies the group algebra
$\C[W]$ with the (degree zero) endomorphisms of $S_\g$.

\medskip

There are two immediate ways in which $S_\g$ and $S_{\CN}$ are related.
First, by proper base change, restriction along the inclusion $i:\CN\hra \g$ induces an identification
$$
i^*S_\g[\dim_\C \fh] \simeq S_{\CN}.
$$
Second, if we identify $\g$ with its dual $\g^*$ via the Killing form, then we can regard the 
shifted Fourier
transforms as endofunctors on perverse sheaves on $\g$. With this understanding,
the shifted Fourier transforms exchange the two perverse sheaves:
$$
(S_\CN)^\w[\dim_\C \g] \simeq S_\g
\qquad
(S_\g)^\vee[-\dim_\C \g] \simeq S_\CN.
$$

\medskip

It is also possible
to construct $S_{\CN}$ as the nearby cycles
of the constant perverse sheaf $\C_{\g_{rs}}[\dim_\C\g]$ in the (multi-dimensional) 
family defined by the adjoint quotient map $q:\g \to \fh\inv W$. 
Namely, we have an identification
$$
R\psi(\C_{\g_{rs}}[\dim_\C\g]) \simeq S_\CN
$$
where $R\psi$ can be taken to be the nearby cycles in the direction of
any line $\mathbb A^1\hra \fh \inv W$ such that $\mathbb A^1\setminus \{0\} \hra \fh_r\inv W$.
In this realization, the braid group $B_W\simeq \pi_1(\fh_r\inv W, pt)$
acts on the nearby cycles by monodromy transformations. This action factors through the projection 
$B_W\to W$ giving another realization of the $W$-action on $S_\CN$.


\subsection{Cotangent bundle of adjoint quotient}

Our aim here is to explain how we will think about the cotangent bundle of 
the adjoint quotient $\g/G$. The discussion is included to elucidate what follows
and is not needed in any technical sense. 

\subsubsection{General formalism}
Suppose the group $G$ acts on a smooth manifold $X$.  
Then the induced action of $G$ on the cotangent bundle $T^*X$ preserves
the canonical one-form $\theta$, and hence the
symplectic form $\omega=d\theta$ as well.

Consider the moment map $m : T^*X \rightarrow \g^*$ for the action of $G$.
It is characterized by two properties: (1) $m$ is $G$-equivariant with respect to the coadjoint action
on $\g^*$,
and (2) its differential $dm$ satisfies the contraction formula
$$
\langle dm, x\rangle = \iota_{\widetilde x} \omega, 
\quad
\mbox{ for $x\in\g$},
$$
where $\widetilde x$ is the vector field on $T^*X$ corresponding to $x\in \g.$
Note that $\mathcal L_{\widetilde x}\theta = 0$ and thus
$
\iota_{\widetilde x}\omega = d \langle \theta,\widetilde x\rangle.
$
Therefore one can say that
$m = \theta$ in the sense that $\langle m, x\rangle = \iota_{\widetilde x}\theta.$

\medskip

Consider the quotient stack $X/G$. By the cotangent stack $T^*(X/G)$, we mean
the result of performing Hamiltonian reduction at the zero moment map value. Namely,
we take the zero-fiber of the moment map $m^{-1}(0)\subset T^*X$, and then pass to the quotient stack
$T^*(X/G) = m^{-1}(0)/G$. 

Two immediate comments are in order.

First, there is no reason that the zero-fiber $m^{-1}(0)$
or that $T^*(X/G)$
should be smooth, and in our case of interest this is not so. 
This should not cause any consternation since we will always be working
in the ambient target $T^*X$. 
To be precise, 
consider the correspondence
$$
\xymatrix{
T^*(X/G) & \ar[l]_-{\ell} T^*(X/G) \times_{X/G} X \ar[r]^-{r} & T^*X
}
$$
where $\ell$ is the obvious smooth projection, and $r$ is the obvious inclusion.
This is nothing more than the Lagrangian correspondence associated to the 
smooth projection $X\to X/G$. 
Note that applying the correspondence to $T^*(X/G)$ itself recovers the zero-fiber of the moment map
$m^{-1}(0) = r(\ell^{-1}(T^*(X/G))$.
In short, all of our concrete geometric arguments will take place in $T^*X$,
and we only keep track of $T^*(X/G)$ to help us understand what is going on.
For example, by a smooth Lagrangian $\CL\hra T^*(X/G)$,
we will mean a substack such that 
$$
L = r(\ell^{-1}(\CL)) \hra T^*X
$$ 
is a smooth Lagrangian. 

Second, just as one uses
the sophisticated technology of stacks to deal with quotients, one should set the moment map $m$
equal to zero
in the appropriate homotopical sense. But since our aims are purely topological,
we can safely ignore this issue and work naively with $T^*(X/G)$ as a stack rather than
as a derived stack. In other words, the reader unaccustomed to this kind of enhancement
can safely ignore the issue, and in particular this comment itself.

\subsubsection{Case of adjoint quotient}

Now let us apply the preceding to the case $X=\g$ with the adjoint action of $G$.

Under the identification of $\g$ with its dual $\g^*$ via the Killing form,
 the moment map becomes the Lie bracket 
$$m:\g\times \g\simeq T^*\g\to \g^*\simeq \g
\qquad
m(x,\xi) = [x,\xi].
$$

Thus the zero-fiber is
the space of commuting pairs
$$
m^{-1}(0) = \{ (x,\xi)\in \g\times\g |  [x,\xi] = 0\}\subset \g\times\g,
$$
and the cotangent bundle is the diagonal adjoint quotient
$$
T^*(\g/G) = \{ (x,\xi)\in \g\times\g |  [x,\xi] = 0\}/G.
$$



\subsection{Quantization of regular Hitchin fibers}

Choose an embedding $\fh\hra  \g$ of the universal Cartan,
and let $H\hra G$ be the corresponding maximal torus.

Recall that $q:\g\to \fh\inv W$ denotes the affine adjoint quotient arising from
Chevalley's identification $\C[\g]^G\simeq \C[\fh]^W$.

We use the phrase Hitchin fibration to refer to the map
$$
\CH:T^*(\g/G) \to \fh\inv W
\qquad
\CH(x,\xi) = q(\xi).
$$
The definition and nomenclature
come from the
observation that the stack $\g/G$ is isomorphic to the stratum of semistable $G$-bundles
on a cuspidal elliptic curve (or equivalently, 
 bundles whose pullback to the normalization $\mathbb P^1$
are trivializable).

\medskip

Recall that $\fh_r\subset \fh$ denotes the $W$-regular locus.
For $\lambda\in \fh_r\inv W$, we refer to the inverse image 
$\CL_\lambda = \CH^{-1}(\lambda)\subset T^*(\g/G)$
as a regular Hitchin fiber.
In terms of our usual identifications, we have the explicit description
$$
\CL_\lambda = \{ (x,\xi)\in \g\times\g | [x,\xi] = 0, q(\xi) = \lambda \}/ G.
$$
Our immediate goal is to show that $\CL_\lambda$ is an exact Lagrangian and carries a canonical brane
structure. (General formalism shows that it is Lagrangian, but we will give an explicit reason
in a moment.)

\medskip

To clarify the discussion, let us denote the Lie algebra by $\g_x$ and its dual by $\g_\xi$.
(We will continue to identify them via the Killing form.)
Now let us switch our perspective, and consider $T^*(\g_x/G)$ as the cotangent
to the coadjoint quotient $\g_\xi/G$. Then by construction,
the Hitchin fiber $\CL_\lambda\hra T^*(\g_x/G)$
is nothing more than the conormal to the coadjoint orbit $\CO_\lambda/G\hra \g_\xi/G$.
Thus it is an exact Lagrangian and carries a canonical brane structure. 
 By this, we mean that the base change
$$
L_\lambda = r(\ell^{-1}(\CL_\lambda)) \subset \g_x\times \g_\xi
$$ 
carries a canonical $G$-equivariant brane structure. 
Namely, as the conormal to a closed submanifold, it comes equipped
with a standard brane structure.

To keep with usual conventions, 
we will shift the standard brane structure of $L_\lambda$ by the amount 
$\dim_\C\fh = \dim_\C\g-\dim_\C\CO_\lambda$.
We use the phrase Hitchin brane to refer to $L_\lambda \hra \g_x\times\g_\xi$ 
with this brane structure.
Thus by construction, the constructible complex 
$\pi_{\g_\xi}(L_\lambda)\in Sh_c(\g_\xi)$ is simply the shifted constant
sheaf $\C_{\CO_\lambda}[-\dim_\C\fh]$.

\medskip

Now 
we will apply Theorem~\ref{main thm}
to the Hitchin brane $L_\lambda$.
Our aim is to identify the constructible complex
$\CF_\lambda=\pi_{\g_x}(L_\lambda)\in Sh_c(\g)$.
Observe that $L_\lambda$ is conic along the factor $\g_x$,
and hence $\CF_\lambda$ is conic as well. We refer to $\CF_\lambda$ as the Hitchin sheaf.

 \begin{thm}
The Hitchin sheaf  $\CF_\lambda$ is isomorphic to the global Springer sheaf $ S_\g$
 \end{thm}

 \begin{proof}
By construction, the constructible complex 
shifted constant sheaf $\C_{\CO_\lambda}[-\dim_\C\fh]$.
Note that $L_\lambda$ is a balanced brane, since its corner consists of projectivized pairs
of commuting nilpotent elements. Thus Theorem~\ref{main thm} provides an identification
$$
\CF_\lambda \simeq (\Upsilon(\C_{\CO_\lambda}[-\dim_\C\fh]))^\w
$$
To complete the proof, we use the well-known identities of Springer theory
$$
\Upsilon(\C_{\CO_\lambda}[\dim_\C\CO_\lambda]) \simeq 
R\psi(\C_{\CO_\lambda}[\dim_\C\CO_\lambda] \simeq S_\CN
\qquad
(S_\CN)^\w[\dim_\C \g] \simeq S_\g.
$$
\end{proof}

The identification $\CF_\lambda\simeq S_\g$ 
is compatible with motions of the parameter $\lambda\in \fh_r\inv W$
as follows.
By construction, the identification
$\pi_{\g_\xi}(L_\lambda)\simeq \C_{\CO_\lambda}[-\dim_\C\fh]$
is compatible with parallel transport with respect to $\lambda\in \fh_r\inv W$.
For this, recall that from the perspective of $\g_\xi$, the brane $L_\lambda\hra T^*\g_\xi$
is nothing more than (a shift of) the standard brane with support the conormal
to $\CO_\lambda\hra \g_\xi$.
Now, under the Fourier transform, motions of $\C_{\CO_\lambda}[-\dim_\C\fh]$
induce the usual Weyl group action on $S_\g$ of Springer theory. 
This is equivalent to the fact that the monodromy braid group action on the nearby cycles
$R\psi(\C_{\CO_\lambda}[-\dim_\C\fh])$ descends to the usual Weyl group action.
Thus in conclusion, the proof of the theorem shows that the braid group action on $\CF_\lambda$
descends to the usual Weyl group action as well.



\bibliographystyle{amsalpha}

\end{document}